\documentclass{amse-new}
\usepackage{color}
\usepackage{soul}
\usepackage{amsthm}
\usepackage{amsmath}
\theoremstyle{plain}
\usepackage{hyperref}
\usepackage{txfonts}
\hypersetup{hypertex=true,
colorlinks=true,
linkcolor=blue,
anchorcolor=blue,
citecolor=blue}
\numberwithin{equation}{section} 

\begin{document}

 \PageNum{1}
 \Volume{201x}{Sep.}{x}{x}
 \OnlineTime{August 15, 201x}
 \DOI{0000000000000000}

\abovedisplayskip 6pt plus 2pt minus 2pt \belowdisplayskip 6pt
plus 2pt minus 2pt
\def\vsp{\vspace{1mm}}
\def\th#1{\vspace{1mm}\noindent{\bf #1}\quad}
\def\proof{\vspace{1mm}\noindent{\it Proof}\quad}
\def\no{\nonumber}
\newenvironment{prof}[1][Proof]{\noindent\textit{#1}\quad }
{\hfill $\Box$\vspace{0.7mm}}
\def\q{\quad} \def\qq{\qquad}
\allowdisplaybreaks[4]


\AuthorMark{Wang Y. }                             

\TitleMark{Tangent measures of self-similar sets}  

\title{Tangent measures of self-similar sets satisfying the strong separation condition        
\footnote{Supported by the Fundamental Research Funds for the Central Universities (Grant No. 60201525020).}
}                 

\author{Yongtao Wang}     
    {School of Science, Northeast Forestry University, Harbin, 150040, Heilongjiang Province, China\\
    E-mail\,$:wangyongtao@nefu.edu.cn$  }

\maketitle%

\Abstract{This paper investigates tangent measures in the sense of Preiss for self-similar sets on ${{\mathbb{R}}^d}$ that satisfy the strong separation condition. Through the dynamics of ``zooming in'' on any typical point, we derive an explicit and uniform formula for the tangent measures associated with this category of self-similar sets on ${{\mathbb{R}}^d}$. Furthermore, for any self-similar set $C\subset{{\mathbb{R}}^d}$ under the open set condition instead of the strong separation condition, we find that the support of any tangent measure at each point $x\in C$ is one of the limit models at that point. Conversely, any limit model at each point $x\in C$
 is the support of one of the tangent measures at that point.}    


\MRSubClass{28A33, 28A78, 28A80}      

\section{Introduction}\label{sec1}
Tangent measures were introduced by David Preiss in \cite{ref11} for the study of sets and measures, and they have turned out to be a powerful tool for studying some geometric properties of measures, e.g., rectifiability. Tangent measures contain information about the local structure of a given Radon measure in a similar but often more complicated way as the derivative of a function tells us about the local behaviour. Since then, scholars have begun to focus on the local structure of more general sets and measures, such as sets and measures with fractional dimensions. 

Over the past thirty years, the research on the local structure of fractals has received 
extensive attention. Many scholars have conducted research by introducing new concepts such as 
tangent measure distributions \cite{ref01,ref02,ref4,Py,Fur,Gav,Hoc1,Hoc2,Fer,Kem,Bar,ref03,ref04,ref05,ref06} and tangent sets \cite{ref3,Buc1,Buc2,Che,Fra1,Fra2,Käe} and have achieved remarkable progress.  
Among them, tangent measure distributions, which describe local behavior from a statistical and probabilistic perspective, have become the mainstream method in current research. They are particularly useful for tasks such as characterizing the uniformly scaling property of fractals \cite{Bar,Fer,Fur,Gav,Hoc1,Hoc2,Kem,Py}. 
In addition, 
tangent sets have intuitive geometric meanings when studying the local structure of fractals and are more easily accepted and applied by researchers. However, tangent measure distributions 
and tangent sets are difficult to effectively describe the local characteristic of each point, especially in non-typical points. 

Currently, there are relatively few works on using tangent measures to study fractals directly, see \cite{reflocal2023} for the latest research work. The main reason is that the construction of tangent measures depends on specific points and particular scaling paths, which makes analysis difficult and lacks a unified explicit expression. Moreover, even if the selected points are the same but the scaling paths are different, it may also lead to the non-uniqueness of tangent measures, thus increasing the complexity and difficulty of analysis.

 This paper investigates tangent measures for self-similar sets on ${{\mathbb{R}}^d}$  satisfying the strong separation condition and derives an explicit and uniform formula. In collaboration with other scholars, a similar problem has been discussed about two types of non-rotational self-similar sets under different separation conditions, namely, homogeneous Cantor sets and non-homogeneous Cantor sets \cite{ref16,ref016,Wan}. The distinction between them lies in whether the Lipschitz constants of the similitudes in the iterated function system that generates them are all equal (see Definition \ref{de:2.17}). 
 In this paper, we allow similitudes to include rotations, which undoubtedly increases the difficulty of the discussion. 
 
 To solve the problem of the explicit expression of tangent measures, we propose a method based on the ``limit models'' (see Definition \ref{de:3.1}), which directly connects the construction of tangent measures with the limit models. Through this method, for some fractals, not only can the explicit expression of the tangent measures at a single point be obtained mathematically, but also the local characteristic of typical and non-typical points can be uniformly processed. Our main result is as follows:
 \begin{theorem}\label{th:1.1}
Let $C\subset{{\mathbb{R}^d}}$ be a self-similar set satisfying the strong separation condition (see Definition \ref{de:2.2} and Definition \ref{de:2.4}).
Let $\mu= {{\cal H}^s}\lfloor_C$ with $s={\dim_H}(C)$. Then for any $x \in C$, we have
   $${\mathop{\rm Tan}\nolimits} (\mu ,x) = \{ {c{\cal H}^s}\lfloor_ F| F \in \omega({C^x}),c >0\},$$
where $\omega({C^x})$ denotes the class of ``tangent set'' of $C$ at $x$ (see Definition \ref{de:3.1}).
\end{theorem}
  \begin{remark}\label{re:1.3}
  Suppose $C\subset{{\mathbb{R}}^d}$ is a self-similar set. If there exists a hyperplane $P$ such that $C \subset P$, we can restrict ourselves to the lower-dimensional space $P$, and repeat until it is no longer contained in any hyperplane. Therefore, without loss of generality, regarding Theorem \ref{th:1.1} we can assume that $C$ is not included in any hyperplane in ${{\mathbb{R}}^d}$.
 \end{remark}
 
It remains unclear whether Theorem \ref{th:1.1} holds under the weaker assumption of the open set condition instead of the strong separation condition, or for more general $s$-regular sets.  However, in the last part of this paper, for any $s$-regular Radon measure $\mu$ on ${{\mathbb{R}}^d}$ with  $C= \operatorname{spt}\mu$ (including any measure $\mu= {{\cal H}^s}\lfloor_C$, where $C\subset{{\mathbb{R}}^d}$ is a $s$-regular set), we show that the support of any tangent measure at each point $x\in C$ must be one of the limit models at that point. Conversely, any limit model at each point $x\in C$
 is the support of one of the tangent measures at that point (see Theorem \ref{th:1.2}).

\begin{theorem}\label{th:1.2}
Given any $s$-regular Radon measure $\mu$ on ${{\mathbb{R}}^d}$ with $C=\operatorname{spt}\mu$ (see Definition \ref{de:5.1}). Then for any $x \in C$, we have
\begin{align}\label{E：1}
\omega({C^x}) =\{\operatorname{spt} \nu|\nu  \in {\mathop{\rm Tan}\nolimits} (\mu,x)\},
\end{align} 
where $\operatorname{spt} {\nu}$ is the smallest closed set A such that
$\nu({{\mathbb{R}^d}}\backslash A) =0$. 
\end{theorem}
 \begin{remark}\label{re:1.4}
  Suppose $C\subset{{\mathbb{R}}^d}$ is a self-similar set satisfying the strong separation condition. Equation \eqref{E：1} can be easily derived from Theorem \ref{th:1.1}.
 \end{remark}

The rest of the paper is organized as follows:
\begin{itemize}
  \item In Section 2, we give the main definitions and fundamental properties, such as self-similar sets and tangent measures.
  \item In Section 3, we provide a detailed discussion of limit models.
  \item In Section 4, we prove Theorem \ref{th:1.1}.
  \item In Section 5, we prove Theorem \ref{th:1.2}.
\end{itemize}
\section{Preliminaries}
The main definitions and basic properties of this section are quoted from Mattila's geometric measure theory. For more details see Mattila \cite{ref10}.
\begin{definition}[\cite{ref10}; 4.13.Self-similar sets]\label{de:2.1}
A mapping $\varphi:{{\mathbb{R}}^d} \to {{\mathbb{R}}^d}$ is called a similitude on ${{\mathbb{R}}^d}$ if there is $r$, $0<r <1$, such that
  \begin{center}
    $\left| {\varphi(x) - \varphi(y)} \right| = r\left| {x - y} \right|$ 
  for $x,y \in {{\mathbb{R}}^d}$.
\end{center}
Similitudes are exactly those maps $\varphi$ which can be written as
\begin{center}
  $\varphi(x) = rg(x) + z$, $x \in {{\mathbb{R}}^d}$
\end{center}
for some $g \in O(d)$, $z \in {{\mathbb{R}}^d}$ and $0 < r < 1$. Here $O(d)$ denotes the set of all orthogonal transformations of ${{\mathbb{R}}^d}$.
\end{definition}
  \begin{definition}[see \cite{ref7,ref1}]\label{de:2.2}
  Let $\{ {\varphi _i}:{{\mathbb{R}}^d} \to {{\mathbb{R}}^d}\} _{i = 1}^m$ be a family of similitudes on ${{\mathbb{R}}^d}$, we say that $\{ {\varphi _i}\} _{i = 1}^m$ is an iterated function system (IFS) on ${{\mathbb{R}}^d}$. The unique non-empty compact set $C \subset {{\mathbb{R}}^d}$ which satisfies the set equation $$C = {\varphi _1}(C) \cup  \cdots  \cup {\varphi _m}(C)$$
 is called the self-similar set generated by the iterated function system (IFS) $\{ {\varphi _i}\} _{i = 1}^m$. (for the existence and uniqueness see \cite{ref7})
 \end{definition}
 \begin{definition}\label{de:2.17}
Let $C\subset{{\mathbb{R}}^d}$ be a self-similar set generated by an IFS $\{ {\varphi _i}= {\beta_i}  x + {c_i}\} _{i = 1}^{m}$, where $0 < {\beta_i}  < 1$ and ${c_i} \in {{\mathbb{R}}^d}$ for $i = 1,2, \cdots ,m$.
 We call $C$ a homogeneous Cantor set if ${\beta _1} = {\beta _2} =  \cdots  = {\beta _{m}}$. Otherwise, we call $C$ a non-homogeneous Cantor set. 
 \end{definition}
 Note that the class of non-homogeneous Cantor sets is a special class of self-similar sets: similitudes in its general IFS admit different radios, and rotations are not allowed.
  \begin{definition}\label{de:2.4}
 Let $C\subset{{\mathbb{R}}^d}$ be a self-similar set generated by an IFS $\{ {\varphi _i}\} _{i = 1}^{m}$. We say that $\{ {\varphi _i}\} _{i = 1}^m$ (or $C$) satisfies the open set condition (OSC) if there is a non-empty open set $O\subset {{\mathbb{R}}^d}$ such that $\mathop  \cup \limits_{i = 1}^m {\varphi _i}(O) \subset O$ and ${\varphi _i}(O) \cap {\varphi _j}(O) = \emptyset $ for $i \ne j =1, \cdots,m$. And we say that $\{ {\varphi _i}\} _{i = 1}^m$ (or $C$) satisfies the strong separation condition (SSC) if the different parts ${\varphi _i}(C)$ are disjoint.
\end{definition}
\begin{definition}\label{de:2.19}
 Let $C\subset{{\mathbb{R}}^d}$ be a self-similar set generated by an IFS $\{ {\varphi _i}= {\beta_i}g_i (x) + {c_i}\} _{i = 0}^{m-1}$, where $0 < {\beta_i}  < 1$, $g_i\in O(d)$ and ${c_i}\in {{\mathbb{R}^d}}$ for $i = 0,1, \cdots ,m - 1$. Let ${\Sigma^ + } = \{ {x_0}{x_1}...:{x_i} = 0,1, \cdots ,m - 1\}$. Given a finite sequence ${x_0} \ldots {x_{p - 1}}$ in $\{ 0,1, \cdots ,m - 1\} $,
  set
  $${\varphi _{{x_0}\ldots {x_{p- 1}}}} ={\varphi _{{x_0}}} \circ  \cdots  \circ {\varphi _{{x_{p - 1}}}}$$
 and
 \[
\begin{aligned}
{C_{{x_0} \ldots {x_{p - 1}}}} = {\varphi _{{x_0} \ldots {x_{p - 1}}}}(C) =& {\beta _{{x_0}}}{\beta _{{x_1}}} \cdots {\beta _{{x_{p - 1}}}}{g_{{x_0}}} \circ{g_{{x_1}}}\circ  \cdots \circ {g _{{x_{p - 1}}}}(C) + \\
& \sum\limits_{i = 0}^{p - 1} {{\beta _{{x_0}}}{\beta _{{x_1}}} \cdots {\beta _{{x_{i - 1}}}}}{g _{{x_0}}}\circ{g _{{x_1}}}\circ \cdots \circ{g _{{x_{i - 1}}}}({{c_{{x_i}}}}).
\end{aligned}
\]
  For a fixed infinite sequence $\mathop x\limits_ -   = {x_0}{x_1}...\in {\Sigma^ + }$, we have $diam$ $ ({C_{{x_0}\ldots{x_{p - 1}}}} ) \to 0$ as $p \to \infty $.
  Hence $\mathop  \cap \limits_{p \ge 1} {C_{{x_0} \ldots {x_{p - 1}}}} $ only contains a single point $x=\sum\limits_{i = 0}^\infty  {{\beta _{{x_0}}}{\beta _{{x_1}}} \cdots {\beta _{{x_{i - 1}}}}}{g _{{x_0}}}\circ{g _{{x_1}}} \circ\cdots\circ {g _{{x_{i - 1}}}}({{c_{{x_i}}}})$ of $C$.\\
  For any $\mathop x\limits_ -   = {x_0}{x_1}...\in {\Sigma^ + }$ 
  and $p\in{{\mathbb{N}}_+}$,
  let ${\lambda ^p}(\mathop x\limits_ -  ) = {\beta _{{x_0}}}{\beta _{{x_1}}} \cdots {\beta _{{x_{p - 1}}}}$, ${h^p}(\mathop x\limits_ -  ) = {g_{{x_0}}}\circ{g _{{x_1}}} \circ\cdots \circ{g _{{x_{p - 1}}}}$, ${\lambda _p}(\mathop x\limits_ -  ) = {({\lambda ^p}(\mathop x\limits_ -))^{ - 1}}$ and ${h _p}(\mathop x\limits_ -  ) = {({h ^p}(\mathop x\limits_ -))^{ - 1}}$. 
  Let ${\lambda ^p}(\mathop x\limits_ -  ) ={\lambda _p}(\mathop x\limits_ -  ) =1$ and  ${h ^p}(\mathop x\limits_ -  )={h _p}(\mathop x\limits_ -  ) =I_d$ when $ p=0$, where $I_d$ is an identity transformation. 
 If a mapping $\pi:\Sigma^+\to C$ satisfies that for each $\mathop x\limits_ -   = {x_0}{x_1}...\in {\Sigma^ + }$, we have
 \begin{align}\label{E：2.2}
x=\pi (\mathop x\limits_ -  ) =\mathop {\lim }\limits_{p \to \infty } {\varphi _{{x_0} \ldots {x_{p - 1}}}}(C)= \sum\limits_{i = 0}^\infty  {{\lambda ^i}(\mathop x\limits_ -){h^i}(\mathop x\limits_ - )({c_{{x_i}}} }),
\end{align}
then $\pi:\Sigma^+\to C$ is called the coding mapping of $C$ with respect to $\{\varphi_i\}_{i = 0}^{m - 1}$.\\
 Since $\mathop  \cup \limits_{i = 0}^{m - 1} {\varphi _i}(C) = C$, we have $C = \{ \pi (\mathop x\limits_ -  ) | {\mathop x\limits_ -   = {x_0}{x_1}... \in {\Sigma^ + }}\} $.
At this time, $\pi$ is a surjective mapping. That is, for each point $x$ in $C$, one can always find an element $\mathop x\limits_ - $ in $\Sigma^+$ corresponding to it. If $\pi(\mathop x\limits_ - ) = x$, then $\mathop x\limits_ - $ is called a coding of $x$. Hence each point of $C$ gets coded in this way.\\
  We define a metric on the space $\Sigma ^ + $ by
  $$
  {d^ * }(\mathop x\limits_ - , \mathop y\limits_ - )=2^{-n} \ \text { where } n=\min \left\{k: x_k \neq y_k\right\}.
  $$
 Suppose $\{ {\varphi _i}\} _{i = 0}^{m-1}$ satisfies the strong separation condition, then $\pi $ is a homeomorphic mapping which implies that every point of  $C$ has a unique coding. 
\end{definition}
\begin{remark}\label{re:2.12}
Suppose $\{ {\varphi _i}\} _{i = 0}^{m-1}$ satisfies the strong separation condition.
   For any $\mathop x\limits_ -   = {x_0}{x_1}... \in {\Sigma^ + }$, we define $\sigma :{\Sigma ^ + } \to {\Sigma ^ + }$ by
  $$\sigma ({x_0}{x_1}{x_2} \cdots ) = {x_1}{x_2} \cdots$$
and $\pi  \circ {\sigma ^n}:{\Sigma ^ + } \to C$ by
$$x_n^* = \pi  \circ {\sigma ^n}(\underline x ),$$
 where
   $ {{\sigma}^n}(\underline x)=\underbrace {\sigma \circ \sigma \circ  \cdots  \circ \sigma}_{n-times}(\underline x )$ for any $n \in {{\mathbb{N}}_+ }$ and $ {{\sigma}^n}(\underline x)=\underline x$ for $n=0$.\\
  Notice that $${\varphi _{{x_0}}}(\pi ({x_1}{x_2}....)) = \pi ({x_0}{x_1}{x_2}....).$$
We define
 $$S:{\varphi _0}(C) \cup {\varphi _1}(C) \cup \cdots \cup {\varphi _{m-1}}(C) \to C $$
 by
\begin{equation}	
 S(x)=\varphi _i^{ - 1}(x)\; if \;x \in {\varphi _i}(C)
 \end{equation}
for $i = 0,1,\cdots,m - 1$. \\
Let $ {S^n}(x)=\underbrace {S \circ S \circ  \cdots  \circ S}_{n-times}(x) $
for any $n \in {{\mathbb{N}}_ + }$ and  $ {S^n}(x)=x$ for $n=0$.\\
 The action of $S$ on  $x$ is equivalent to eliminating one bit of the coding of $x$ from the left. That is,  $${S^n}(x)=\pi  \circ {\sigma ^n} \circ {\pi}^{-1} (x)$$
 for any $n \in {{\mathbb{N}} }$.\\
Let $\lambda _r^{^p}(\mathop x\limits_ -) = {{{\lambda ^p}(\mathop x\limits_ -)} \over {{\lambda ^r}(\mathop x\limits_ -)}}$
and
 $h _r^{^p}(\mathop x\limits_ -) = {({{h ^r}(\mathop x\limits_ -)})^{-1}}\circ {{h ^p}(\mathop x\limits_ -)} $
for any $r, p \in {{\mathbb{N}} }$. Then for any $\mathop x\limits_ -  \in {\Sigma^ + }$ with $\pi (\mathop x\limits_ -) = x$,
we have
\begin{align}\label{E：2.4}
x_n^* ={S^n}(x) = \sum\limits_{i = 0}^\infty  {\lambda _n^{^{i + n}}}(\mathop x\limits_ -){h _n^{^{i + n}}}(\mathop x\limits_ -)({c_{{x_{i + n}}}}).
\end{align}
\end{remark}
  \begin{theorem}[\cite{ref10}; 4.14.Theorem]\label{th:2.5}
  Let $C\subset{{\mathbb{R}}^d}$ be a self-similar set generated by an IFS $\{ {\varphi _i}\} _{i = 1}^{m}$, each with ratio ${r_i}=Lip({\varphi _i})$. If $\{ {\varphi _i}\} _{i = 1}^m$ satisfies the open set condition, then $0 < {{\cal H}^s}(C) < \infty $, whence $s={\dim _H}(C)$, where $s$ is the unique number for which
  \begin{equation}
  \sum\limits_{i = 1}^m {r_i^s}  = 1.
  \end{equation}
  Moreover, there are positive and finite numbers $a$ and $b$ such that $a{(2r)^s} \le {{\cal H}^s}(C \cap B(x,r)) \le b{(2r)^s}$ for $x \in C, 0 < r \le 1$.
 \end{theorem}
 
  \begin{definition}[\cite{ref10}\label{de:2.7}; 6.8.Definition]
  Let $0 \le s < \infty $, $A \subset {{\mathbb{R}}^d}$ and $a \in {{\mathbb{R}}^d}$. The upper and lower s-densities of $A$ at $a$ are defined by \[{\theta ^{ * s}}(A,a) = \mathop {\lim \sup }\limits_{r \to 0}  {(2r)^{ - s}}{{\cal H}^s}(A \cap B(a,r)),\] \[\theta _ * ^s(A,a) = \mathop {\lim \inf }\limits_{r \to 0} {(2r)^{ - s}}{{\cal H}^s}(A \cap B(a,r)).\]
  If they agree, their common value is called the s-dimensional density of $A$ at $a$ and denoted by \[{\theta ^s}(A, a) = {\theta ^{ * s}}(A, a) = \theta _ * ^s(A, a).\]
  \end{definition}
\begin{remark}[\cite{ref10}; 6.4.Remark(5)]\label{re:2.8}
 As noted in Theorem \ref{th:2.5} more can be said about the densities of a self-similar set C with the open set condition. In particular, there are positive and finite numbers $a$ and $b$ such that $0 < a \le \theta _ * ^s(C,x) \le {\theta ^{ * s}}(C,x) \le b < \infty $ for $x \in C$.
\end{remark}
\begin{definition}[\cite{ref10}; 1.21.Definition]\label{de:2.9}
Let $\mu ,{\mu _1},{\mu _2}, \ldots $ be Radon measures on ${{\mathbb{R}}^d}$. We say that a sequence $\{ {\mu _n}\} $ converges weakly to $\mu $, denoted by
  $${\mu _n}\mathop  \to \limits^w \mu ,$$ if $$\mathop {\lim }\limits_{n \to \infty } \int {\varphi d} {\mu _n} = \int {\varphi d} \mu \; for \; all \; \varphi  \in {C_c}({{\mathbb{R}}^d}),$$
  where ${C_c}({{\mathbb{R}}^d})$ is the space of compactly supported continuous real-valued functions on ${{\mathbb{R}}^d}$.
\end{definition}
\begin{definition}[\cite{ref10}; 14.1.Definition]\label{de:2.10}
Let $\mu $ be a Radon measure on ${{\mathbb{R}}^d}$. We say that $\nu$ is a tangent measure of $\mu $ at a point $a \in {{\mathbb{R}}^d}$ if there exist sequences $\{ {r_i}\} $ and $\{ {c_i}\} $ of positive numbers such that ${r_i} \to 0$ and $${c_i}{{{\mu _{a,{r_i}}}}} \mathop  \to \limits^w \nu \; as \; i \to \infty ,$$
  where ${\mu _{a,{r_i}}}(A) = \mu (a + {r_i}A)$ for any $\mu $-measurable set  $A \subset {{\mathbb{R}}^d}$.
\end{definition}
The set of all such tangent measures is denoted by $\operatorname{Tan} (\mu, a)$.

\begin{lemma}[\cite{ref016}; Lemma 2.12]\label{le:2.12}
For a self-similar set $C\subset{{\mathbb{R}^d}}$ with the open set condition and each $x \in C$, let $\mu  = {{\cal H}^s} \lfloor _C$ with $s = \dim C$, then
\[
\begin{aligned}
\operatorname{Tan} (\mu ,x) &= \{ \nu: \exists \; {r_i} \to {\rm{0}}, c>0\; such \; that\; c\mathop {{{\mu _{a,{r_i}}}} \over {{r_i}^s}}\mathop  \to \limits^w \nu\}\\
&=\{ \nu: \exists \; {r_i} \to {\rm{0}}, c>0\; such \; that\; c\mathop {{\cal H}^s}\lfloor _{{{{C^x}} \over {{r_i}}}}\mathop  \to \limits^w \nu\}.
\end{aligned}
\]
\end{lemma}

\section{Properties of limit models in ${{\mathbb{R}}^d}$}

 Inspired by the limit models of Tim Bedford and Albert M. Fisher, we apply limit models (see Definition \ref{de:3.1}) to any closed set in ${{\mathbb{R}}^d}$.
  In this section, we prove some properties of limit models. For example, Corollary \ref{co:3.4} will give an equivalent definition of limit models of self-similar sets in ${{\mathbb{R}}^d}$. Some properties for Hausdorff measures on limit models of self-similar sets in ${{\mathbb{R}}^d}$ under the strong separation condition are given in Proposition \ref{pr:5.2} and its corollaries.  
  In the proof of Proposition \ref{pr:5.2} and its corollaries, Lemma \ref{le:3.61}, Lemma \ref{le:3.7},  Lemma \ref{le:4.} and Lemma \ref{le:5.} play a key role.

 \subsection{General definition and facts in ${{\mathbb{R}^d}}$}
 
\begin{definition}[see \cite{ref9}]\label{de:2.13}
 Let $\Gamma  = \{ F \subset {{\mathbb{R}}^d}:F$ is closed${\rm{\}}}$ (including the empty set $\emptyset $).
 The notations $ \mathcal{G}({{\mathbb{R}}^d})$ and $\mathcal{K}({{\mathbb{R}}^d})$, or, if there is no ambiguity simply $\mathcal{G}$ and $\mathcal{K}$ denote the classes of the subsets respectively open and compact in ${{\mathbb{R}}^d}$. 
 Let
   $$ {\Gamma ^B} = \{ F:F \in \Gamma, F \cap B = \emptyset \},$$
 $$
   {\Gamma _B} = \{ F:F \in \Gamma, F \cap B \ne \emptyset \}$$
 for any subset $B$ of ${{\mathbb{R}}^d}$.
Throughout the article, the space $\Gamma$ is topologized by the topology $\rho$, generated by the two families ${\Gamma}^K$, $K\in \mathcal{K}$, and ${\Gamma}_{G}$, $G\in \mathcal{G}$.
What's more, by putting
\begin{equation}\label{eq:3.1} 
\Gamma _{{G_1},{G_2}, \cdots ,{G_n}}^K = {\Gamma ^K} \cap {\Gamma _{{G_1}}} \cap  \cdots  \cap {\Gamma _{{G_n}}}
\end{equation}
for $K \in \mathcal{K}, n$ integer $\geqslant 0$ and $G_1, \ldots, G_n \in \mathcal{G}$, the corresponding class of subsets of $\Gamma$ is a base for the topology $\rho$ on $\Gamma$. 
Notice that $\Gamma _G^\emptyset = \Gamma_G$ and for $n=0$ the set \eqref{eq:3.1} is ${\Gamma}^K$, 
so that the sets ${\Gamma}^K$, $K\in \mathcal{K}$, and ${\Gamma}_{G}$, $G\in \mathcal{G}$, belong to this base,
as well as ${\Gamma ^\emptyset } = \Gamma $ and ${\Gamma _\emptyset } = \emptyset$. The corresponding topology $\rho$ is known as the Hit-and-Miss topology. The topology space $(\Gamma,\rho)$ has particularly good properties, see Remark \ref{re:2.14}, Lemma \ref{le:3.4} and Lemma \ref{Le:3.5}.
\end{definition}

\begin{remark}[see \cite{ref9}; Theorem 1-2-1]\label{re:2.14} 
  The topology space $(\Gamma,\rho)$ is compact, Hausdorff, and admits a countable base. Hence it is also sequentially compact.
\end{remark}

\begin{lemma}[see \cite{ref9}; Theorem 1-2-2]\label{le:3.4}
A sequence $\{ {F_n}\} \subset \Gamma $ converges to $F$ in $ \Gamma$ by $\rho $ (we might as well write it down as ${F_n}\mathop  \to \limits^\rho  F$ or $F = (\rho )\mathop {\lim }\limits_{n \to \infty } {F_n}$) if and only if the following two conditions hold:

(i). For any $x \in F$, there exists a sequence $\{ {x_n}\} $ such that ${x_n} \in {F_n}$ for each $ n \in {\mathbb{N_+}}$ and $\mathop {\lim }\limits_{n \to \infty } {x_n} = x$.

(ii). For any subsequence ${\{ {n_k}\} _{k \in {\mathbb{N_+}}}}$ and any sequence ${\{x_{{n_k}}\} _{k \in {\mathbb{N}}}}$ with ${x_{{n_k}}} \in {F_{{n_k}}}$, if  $\{x_{{n_k}}\}$ converges, then $\mathop {\lim }\limits_{k \to \infty } {x_{{n_k}}} \in F$.
\end{lemma}

\begin{lemma}\label{Le:3.5} 
Suppose ${F_n}\mathop\to \limits^\rho F$ for any ${F_n},F \in \Gamma $. Then for any non-empty bounded open set $U\subset{{\mathbb{R}}^d}$,  there exists a  closed set ${F^ * }$ and a subsequence $\{ {n_k}\} $
  such that
  $$\rho( {F_{{n_k}}} \cap  \overline U,{F^ * })\to 0,$$
  $$F \cap U  \subset {F^ * } \subset F \cap \overline U$$
and
$$F \cap U= {F^*} \cap U,$$
where $\overline U$ is the closure of $U$.
\end{lemma}
\begin{prof}
  Since ${F_{n}} \cap  \overline U\subset \Gamma $ and $(\Gamma,\rho)$ is  sequentially compact by Remark \ref{re:2.14}, then there is a closed set ${F^ * }$ and a subsequence $\{ {n_{k}}\} $ such that
$$\rho( {F_{{n_k}}} \cap  \overline U,{F^ * })\to 0.$$
Since ${F_{n}}\mathop  \to \limits^\rho  F$, we have
$${F_{{n_k}}}\mathop  \to \limits^\rho  F.$$
By Lemma \ref{le:3.4} ($i$), for any $ y \in F \cap U$, there exists a sequence ${x_{{n_{k}}}} \in {F_{n_k}}$ such that $\mathop {\lim }\limits_{k \to \infty } {x_{{n_{k}}}} = y$.
When $k$ is big enough, we have ${x_{{n_{k}}}} \in {F_{{n_k}}} \cap  \overline U$.
As a result, there exists a sequence ${x_{{n_{k}}}} \in {F_{{n_k}}} \cap  \overline U$ such that $\mathop {\lim }\limits_{k \to \infty } {x_{{n_{k}}}} = y$.
By Lemma \ref{le:3.4} ($ii$), we have $y\in {F^*}$.
Hence
$$F \cap U\subset {F^*}.$$
Similarly, since ${F_{{n_k}}} \cap \overline U\mathop  \to \limits^\rho  {F^*}$ and ${F_{{n_k}}}\mathop  \to \limits^\rho  F$, it can be concluded that  ${F^*} \subset F$ by Lemma \ref{le:3.4}.\\
Since $\rho( {F_{{n_k}}} \cap  \overline U,{F^ * })\to 0$, we have ${F^*} \subset  \overline U$.\\
Hence
$${F^*} \subset F \cap  \overline U.$$
As a result, we have
$$F \cap U\subset {F^*} \subset F \cap  \overline U.$$
So we have
 $$F \cap U= {F^*} \cap U.$$
 \qed
\end{prof}

  \begin{definition}
[\cite{ref10}\label{de:3.6}; the orthogonal group 3.5.] The orthogonal group $O(d)$ consists of all linear maps $g: \mathbb{R}^d \rightarrow \mathbb{R}^d$ preserving the inner product,
$$
g(x) \cdot g(y)=x \cdot y \quad \text { for all } x, y \in \mathbb{R}^d
$$
or equivalently preserving the distance,
$$
|g(x)-g(y)|=|x-y| \quad \text { for all } x, y \in \mathbb{R}^d.
$$
(The equivalence is easy to check.) Then $O(d)$ is a compact subspace of the metric space of all linear maps $\mathbb{R}^d \rightarrow \mathbb{R}^d$ equipped with the usual metric
$$
\tilde d(g, h)=\|g-h\|=\mathop {\sup }\limits_{|x| \ne 0} {{|g(x) - h(x)|} \over {\left| x \right|}}=\sup _{|x|=1}|g(x)-h(x)|.
$$
\end{definition}
Note that $g(x-y)=g(x)-g(y)$ and $g(kx)=kg(x)$ for any $x,y\in {{\mathbb{R}^d}}$, any $k\in {\mathbb{R}}$ and any $g\in O(d)$.

\begin{proposition}\label{pr:5}
For any $g_n,g \in O(d)$ satisfying  $g_n\mathop  \to \limits^{\tilde d}g$, then\\
(1). $d(g_n(x),g(x))\to 0$ for any $x\in {{\mathbb{R}^d}}$.\\
(2). $d(g_n({x_n}),g(x))\to 0$ for any ${x_n},x\in {{\mathbb{R}^d}}$ satisfying ${x_n}\to x$.\\
(3). $\rho (g_n(G^{x_n}),g(G^x))  \to 0$ for any nonempty compact set $G\in \Gamma$ and any ${x_n},x\in {{\mathbb{R}^d}}$ satisfying ${x_n}\to x$, where ${G^x} = \{ y - x\left| { y \in G\} } \right.$.
\end{proposition}
\begin{prof}
(1). If $x=0$, then $d(g_n(x),g(x))=0$. If $x\ne0$, since $$
\tilde d(g_n, g)=\|g_n-g\|=\mathop {\sup }\limits_{|x| \ne 0} {{|g_n(x) - g(x)|} \over {\left| x \right|}}\to 0,
$$
we have $d(g_n(x),g(x))\to 0$.\\
(2). Since ${x_n}\to x$ and $g_n\mathop  \to \limits^{\tilde d}g$, we have
\begin {align*}
d(g_n({x_n}),g(x))\le 
&\left| {{g_n}({x_n}) - {g_n}(x)} \right| + \left| {{g_n}(x) - g(x)} \right| \\
=& \left| {{x_n} - x} \right| + \left| {{g_n}(x) - g(x)} \right| \to 0.
\end{align*}
(3). By Lemma \ref{le:3.4}, we can obtain
$\rho (g_n(G^{x_n}),g(G^x))  \to 0$. The proof is as follows.\\
On the one hand, 
for any $y\in G$ with $g(y - x)\in g(G^x)$, then 
$g_n(y-x_n)\in g_n(G^{x_n})$
and
$$\left| {{g_n}(y - {x_n}) - g(y - x)} \right| \le \left| {{g_n}(y) - g(y)} \right| + \left| {{g_n}({x_n}) - g(x)} \right| \to 0.$$
On the other hand, for any $y_{n_k}\in G$ with $g_{n_k}(y_{n_k}-x_{n_k})\in g_{n_k}(G^{x_{n_k}})$, suppose $\{g_{n_k}(y_{n_k}-x_{n_k})\}$ is a convergent sequence. Since $G$ is compact, there exists a convergent sub-sequence $\{y_{n_{k_l}}\}\subset \{y_{n_k}\}$ and a positive number $y\in G$ such that
$y_{n_{k_l}}\to y$.\\
By Proposition \ref{pr:5} (2), we have 
$$\left| {{g_{{n_{{k_l}}}}}({y_{{n_{{k_l}}}}} - {x_{{n_{{k_l}}}}}) - g(y - x)} \right| \le \left| {{g_{{n_{{k_l}}}}}({y_{{n_{{k_l}}}}}) - g(y)} \right| + \left| {{g_{{n_{{k_l}}}}}({x_{{n_{{k_l}}}}}) - g(x)} \right| \to 0.$$
Since $\{g_{n_k}(y_{n_k}-x_{n_k})\}$ is a convergent sequence,
 $g_{n_k}(y_{n_k}-x_{n_k})\to g(y - x)\in g(G^x)$.
Hence $$\mathop {\lim }\limits_{k \to \infty } {g_{{n_k}}}({y_{{n_k}}} - {x_{{n_k}}}) \in g({G^x}).$$
This completes the proof. \qed
\end{prof}

\begin{definition}[see \cite{ref3,ref16,ref016,Wan}]\label{de:3.1}
Suppose $G \subset {{\mathbb{R}}^d}$ is a non-empty closed set. For each $x \in G$, let ${G^x} = \{ y - x\left| { y \in G\} } \right.$. The $\omega$-limit set of ${G^x}$ (or $G$ at the point $x$) is defined by $ \omega({G^x}) = \{ F \in \Gamma :\exists$ positive number $ {t_k} \to \infty $ such that $ {e^{{t_k}}}{G^x}\mathop  \to \limits^\rho  F\} $ (or equivalently defined as $ \omega({G^x}) = \{ F \in \Gamma :\exists$ positive number $ {r_k} \to 0 $ such that $ {{{G^x}} \over {{r_k}}}\mathop  \to \limits^\rho  F\} $). If $F \in \omega({G^x})$ for some $x\in G$, then we call $F$ a limit model for $G$. The set $\Omega  = \Omega (G) = \mathop  \cup \limits_{x \in G} \omega({G^x})$ is called the set of limit models for $G$.  

By Remark \ref{re:2.14}, we know that $\omega({G^x}) \ne \emptyset$. By Lemma \ref{le:3.4}, if $ \{{e^{{t_k}}}{G^x}\} $ converges by $\rho$, then the convergence limit of $ \{{e^{{t_k}}}{G^x}\} $ is unique. In addition, for any $F\in\omega({G^x})$, we have $0\in F$. Let $T = \{F\in\Gamma:0\in F\}$, then $\Omega\subset T$, and it is not difficult to prove that $T$ is a compact set.
\end{definition}

\begin{proposition}\label{pr:3.3}
  Given any  non-empty closed set $G \subset {{\mathbb{R}}^d}$ and any IFS $\{ {\varphi _i}= {\beta_i}g_i (x) + {c_i}\} _{i = 0}^{m-1}$, where $0 < {\beta_i}  < 1$, $g_i\in O(d)$ and ${c_i}\in {{\mathbb{R}^d}}$ for $i = 0,1, \cdots ,m - 1$, then for any $y\in G$ and any $\mathop x\limits_ -   = {x_0}{x_1}...$ $\in {\Sigma^ + }$, we have
  \begin{center}
    $\omega({G^y}) = \{ F \in \Gamma :\exists$ $ \alpha  > 0,$ positive number ${n_k} \to \infty $ such that $\alpha {{\lambda _{{n_k}}}(\mathop x\limits_ - )}{G^y}\mathop  \to \limits^\rho  F\} $.
\end{center}
  \end{proposition}
\begin{prof}
  Fix $y$ and $\mathop x\limits_ - $ as in the statement.\\
  On the one hand, for any $ F \in \omega({G^y}) $, there exists positive number
${t_k} \to \infty $ such that ${e^{{t_k}}}{G^y}\mathop  \to \limits^\rho  F$.\\
 For each $k$, there exists a unique integer ${n_k}\in {\mathbb{N_+}}$ such that $${e^{{t_k}}} \in [{{\lambda _{{n_k}}}(\mathop x\limits_ - )},{{\lambda _{{n_{k}}+1}}(\mathop x\limits_ - )}).$$
 As a result, $\{ \frac{{{e^{{t_k}}}}}{{{{{\lambda _{{n_k}}}(\mathop x\limits_ - )}}}}\} $ is a bounded sequence. Hence there exists a convergent subsequence,  still denoted by  $\{ \frac{{{e^{{t_k}}}}}{{{{\lambda _{{n_k}}}(\mathop x\limits_ - )}}}\} $ for short, and a positive number $\alpha $ such that
 $$\mathop {\lim }\limits_{k \to \infty } \frac{{{e^{{t_k}}}}}{{{{{\lambda _{{n_k}}}(\mathop x\limits_ - )}}}} =\alpha \in [1,{\beta _{\min }} ^{-1}],$$
 where ${\beta _{\min }} = \mathop {\min }\limits_{0 \le i \le m - 1} ({\beta _i})$.\\
  Since ${e^{{t_k}}}{G^y}\mathop  \to \limits^\rho  F$, by Lemma \ref{le:3.4} we easily have $\alpha {\beta _{{n_k}}}(\mathop x\limits_ - ){G^y}\mathop  \to \limits^\rho  F$, $\alpha  \in [1,{\beta _{\min }}^{-1}]$. \\
  On the other hand, if $\alpha {\lambda _{{n_k}}}(\mathop x\limits_ - ){G^y}\mathop  \to \limits^\rho  F$, let ${t_k} = \ln (\alpha {\lambda _{{n_k}}}(\mathop x\limits_ - ))$, we have ${e^{{t_k}}}{G^y}\mathop  \to \limits^\rho  F$.\\
  As a result, we have 
  $$\omega({G^y}) = \{ F \in \Gamma :\exists\; \alpha  > 0, \;positive\; number \;{n_k} \to \infty \; such \; that \; \alpha {\lambda_{{n_k}}}(\mathop x\limits_ - ){G^y}\mathop  \to \limits^\rho  F\} .$$
  \qed
  \end{prof}

  \begin{corollary}\label{co:3.4} 
  Let $C\subset{{\mathbb{R}}^d}$ be a self-similar set generated by an IFS $\{ {\varphi _i}= {\beta_i}g_i (x) + {c_i}\} _{i = 0}^{m-1}$, where $0 < {\beta_i}  < 1$, $g_i\in O(d)$ and ${c_i}\in {{\mathbb{R}^d}}$ for $i = 0,1, \cdots ,m - 1$. For any $\mathop x\limits_ -   = {x_0}{x_1}...$ $\in {\Sigma^ + }$, we have
  \begin{center}
    $\omega({C^x}) = \{ F \in \Gamma :\exists$ $ \alpha  > 0,$ positive number ${n_k} \to \infty $ such that $\alpha {{\lambda _{{n_k}}}(\mathop x\limits_ - )}{C^x}\mathop  \to \limits^\rho  F\} $,
\end{center}
 where $x=\pi (\mathop x\limits_ -  )$.
\end{corollary}
\begin{prof}
One can check easily that it can be proved by Proposition \ref{pr:3.3}.
  \qed
  \end{prof}
\subsection{Limit models of self-similar sets in ${{\mathbb{R}}^d}$}

\begin{lemma}\label{le:3.61}
Let $C\subset{{\mathbb{R}}^d}$ be a self-similar set generated by an IFS $\{ {\varphi _i}= {\beta_i}g_i (x) + {c_i}\} _{i = 0}^{m-1}$, where $0 < {\beta_i}  < 1$, $g_i\in O(d)$ and ${c_i}\in {{\mathbb{R}^d}}$ for $i = 0,1, \cdots ,m - 1$. Suppose $\{ {\varphi _i}\} _{i = 0}^{m-1}$ satisfies the strong separation condition. Let $\mathop \Pi \limits_{i = 1}^d [{a_i},{b_i}] = [{a_1},{b_1}] \times  \cdots  \times [{a_d},{b_d}]$, $\mathop \Pi \limits_{i = 1}^d ({a_i},{b_i})= ({a_1},{b_1}) \times  \cdots  \times ({a_d},{b_d})$.
For any $\mathop x\limits_ -   = {x_0}{x_1}...$ $\in {\Sigma^ + }$, any $F =(\rho ) \mathop {\lim }\limits_{k \to \infty } \alpha {{\lambda _{{n_k}}}(\mathop x\limits_ -)}{C^x} \in \omega({C^x})$ with $\pi (\mathop x\limits_ -) = x$, and any $ {a_i},{b_i} \in {\mathbb{R}}$, ${a_i} < {b_i}$, $i = 1,2, \ldots ,d$, \\
(1). There is a closed set ${F^ * }$ and a subsequence $\{ {n_{k_l}}\} $ of $\{ {n_k}\} $
  such that
  $$\rho( (\alpha {\lambda _{{n_{{k_l}}}}(\mathop x\limits_ -)}{C^x}) \cap \mathop \Pi \limits_{i = 1}^d [{a_i},{b_i}],{F^ * })\to 0$$
and
$$F \cap \mathop \Pi \limits_{i = 1}^d ({a_i},{b_i})= {F^*} \cap \mathop \Pi \limits_{i = 1}^d ({a_i},{b_i}).$$
(2). For any positive integer $N$, there exists a integer $r(N)$ such that for each positive integer $w \ge -r(N)$,
$$({\lambda _w}(\mathop x\limits_ -){C^x}) \cap \mathop \Pi \limits_{i = 1}^d [ - N,N] = (\lambda _w^{w + r(N)}(\mathop x\limits_ -){h^{w +r(N)}}(\mathop x\limits_ -)({C^{x_{w +r(N)}^*}} ))\cap \mathop \Pi \limits_{i = 1}^d [ - N,N] ,$$
where $x_{w +r(N)}^* = {S^{w +r(N)}}(x)$.\\
(3). If $ - \alpha  N <  {a_i}< {b_i} <  \alpha N, i = 1,2, \ldots ,d $ for a positive integer $ N $, then for the $r(N)$ in Lemma \ref{le:3.61} (2), 
there exists a sub-sequence  $\{n_{k_l}\} $ of $\{n_k\} $, $h(N) \in O(d)$, ${z_N} \in C$ and $\beta (N) \in [0,\mathop \beta \limits^ -  ]$ with $\mathop \beta \limits^ -  = max\{ {({\beta _{\max }})^{r(N)}},{({\beta _{\min }})^{r(N)}}\}$ such that
  $$(\alpha{{{\lambda _{n_{k_l}}}(\mathop x\limits_ -)}}{C^x}) \cap \mathop \Pi \limits_{i = 1}^d [{a_i},{b_i}] =(\alpha{\beta (N)}{h^{n_{k_l} +r(N)}}(\mathop x\limits_ -) ({{C^{x_{n_{k_l}+r(N)}^*}}}))\cap \mathop \Pi \limits_{i = 1}^d [{a_i},{b_i}] ,$$
  $$\tilde d({h^{n_{k_l} +r(N)}}(\mathop x\limits_ -),h(N)) \to 0,$$
 $$x_{n_{k_l}+r(N)}^* \to {z_N}$$ 
  and
$$ F \cap \mathop \Pi \limits_{i = 1}^d 
({a_i},{b_i})= (\alpha {\beta (N)}h(N)({C^{{z_N}}})) \cap \mathop \Pi \limits_{i = 1}^d 
({a_i},{b_i}).$$
\end{lemma}
 \begin{prof}
  (1). For any $\mathop x\limits_ -   = {x_0}{x_1}...$ $\in {\Sigma^ + }$, any $F =(\rho ) \mathop {\lim }\limits_{k \to \infty } \alpha {{\lambda _{{n_k}}}(\mathop x\limits_ -)}{C^x} \in \omega({C^x})$ with $\pi (\mathop x\limits_ -) = x$, and any $ {a_i},{b_i} \in {\mathbb{R}}$, ${a_i} < {b_i}$, $i = 1,2, \ldots ,d$, by Lemma \ref{Le:3.5}, there is a  closed set ${F^ * }$ and a subsequence $\{ {n_{k_l}}\} $ of $\{ {n_k}\} $
  such that
  $$\rho( (\alpha {\lambda _{{n_{{k_l}}}}(\mathop x\limits_ -)}{C^x}) \cap \mathop \Pi \limits_{i = 1}^d [{a_i},{b_i}],{F^ * })\to 0$$
and
$$F \cap \mathop \Pi \limits_{i = 1}^d ({a_i},{b_i})= {F^*} \cap \mathop \Pi \limits_{i = 1}^d ({a_i},{b_i}).$$
  (2). For any $y \in C$, $y \ne x$,
let  $\mathop {{\rm{ }}y}\limits_ -   = {y_0}{y_1}...$ be the coding of  $y$.
  For any positive integer $N$, if ${\lambda_{{w}} }(\mathop x\limits_ -) (y - x)  \in B(0,N\sqrt d)$, first notice that there exists an integer $r(N)$ such that for each positive integer $w \ge -r(N)$,
\begin{align}\label{E：3.2}
{y_r} = {x_r}\; for\; any\; r< w+r(N).
\end{align}
 Hence by \eqref{E：2.2}, \eqref{E：2.4} and  \eqref{E：3.2}, we have
 \begin{equation*}
 \begin{split}
     &({\lambda _w}(\mathop x\limits_ -){C^x}) \cap \mathop \Pi \limits_{i = 1}^d [ - N,N] \\=&({\lambda _w}(\mathop x\limits_ -){C^x})\cap B(0,N\sqrt d)\cap \mathop \Pi \limits_{i = 1}^d [ - N,N]\\
     =&({\lambda _w}(\mathop x\limits_ -)\{ {y-x}|{y_r} = {x_r},r< w+r(N),y \in C\})\cap \mathop \Pi \limits_{i = 1}^d [ - N,N]\\
     =&(\lambda _w^{w +r(N)}(\mathop x\limits_ -){h^{w +r(N)}}(\mathop x\limits_ -)(\{{\lambda _{w +r(N)}}(\mathop x\limits_ -){h _{w +r(N)}}(\mathop x\limits_ -) (y-x)|{y_r} = {x_r},r< w+r(N),y \in C\}))\\
     &\cap \mathop \Pi \limits_{i = 1}^d [ - N,N]\\
      =&(\lambda _w^{w +r(N)}(\mathop x\limits_ -){h^{w +r(N)}}(\mathop x\limits_ -)(\{{y_{w +r(N)}^*}-{x_{w +r(N)}^*}|{y_{w +r(N)}^*} \in C\}))\cap \mathop \Pi \limits_{i = 1}^d [ - N,N]\\
     =&(\lambda _w^{w +r(N)}(\mathop x\limits_ -){h^{w +r(N)}}(\mathop x\limits_ -)({C^{x_{w +r(N)}^*}})) \cap \mathop \Pi \limits_{i = 1}^d [ - N,N],
 \end{split}
 \end{equation*}
where $x_{w +r(N)}^* = {S^{w +r(N)}}(x)$.\\
Below we provide the proof of \eqref{E：3.2}.\\
 Let
${r_1} = \min \{ i \in {{\mathbb{N}}}\left| {{y_i} \ne {x_i}} \right.\} ,$
we have
${\lambda ^{{r_1}}}(\mathop x\limits_ -) = {\lambda ^{{r_1}}}(\mathop y\limits_ -)$
and
${h ^{{r_1}}}(\mathop x\limits_ -) = {h ^{{r_1}}}(\mathop y\limits_ -)$.\\
 Since $\{ {\varphi _i}\} _{i = 0}^{m-1}$ satisfies the strong separation condition, we have
  \begin{center}
  $d({\varphi _i}(C),{\varphi _j}(C)) > 0$ for any $i \ne j \in\{ 0,1, \cdots ,m - 1\},$
  \end{center}
  where $d(A,B) = \inf \{ d(x,y):x \in A,y \in B\} $ for any non-empty set $A$, $B$ on ${{\mathbb{R}^d}}$.\\
  Let
$$t = \min \{ d({\varphi _i}(C),{\varphi _j}(C)) \left| \right.i \ne j \in\{ 0,1, \cdots ,m - 1\}\}> 0 .$$
Then we have
\[
\begin{aligned}
\left| {x - y} \right| &\ge d({\varphi _{{x_0} \ldots {x_{r_1}}}}(C),{\varphi _{{y_0}\ldots {y_{r_1}}}}(C))\\
& = d({\lambda ^{{r_1}}}(\mathop x\limits_ -){h ^{{r_1}}}(\mathop x\limits_ -)({\varphi _{{x_{{r_1}}}}}(C)),{\lambda ^{{r_1}}}(\mathop x\limits_ -){h ^{{r_1}}}(\mathop x\limits_ -)({\varphi _{{y_{{r_1}}}}}(C))) \\
& = {\lambda ^{{r_1}}}(\mathop x\limits_ -)d({h ^{{r_1}}}(\mathop x\limits_ -)({\varphi _{{x_{{r_1}}}}}(C)),{h ^{{r_1}}}(\mathop x\limits_ -)({\varphi _{{y_{{r_1}}}}}(C))) \\
& = {\lambda ^{{r_1}}}(\mathop x\limits_ -)d({\varphi _{{x_{{r_1}}}}}(C),{\varphi _{{y_{{r_1}}}}}(C)) \\
& \ge {\lambda ^{{r_1}}}(\mathop x\limits_ -)t.
\end{aligned}
\]
Let
${\beta _{\max }} = \mathop {\max }\limits_{0 \le i \le m - 1} ({\beta _i})$
and
${\beta _{\min }} = \mathop {\min }\limits_{0 \le i \le m - 1} ({\beta _i}).$\\
For any $w \in {{\mathbb{Z}}}$, we have
\begin{equation}	
 		{\lambda_{{w}} }(\mathop x\limits_ -)\left| {x - y} \right| \ge t\lambda _w^{{r_1}}(\mathop x\limits_ -)\ge
 		\begin{cases}
 			t{({\beta _{\min }})^{{r_1} - w}},w<{r_1} \\
 			t{({\beta _{\max }})^{{r_1} - w}},w>{r_1}\\
            t,w={r_1} .
 		\end{cases}
 		\label{eq_h_s}
 	\end{equation}
Hence ${\lambda_{{w}} }(\mathop x\limits_ -)\left| {x - y} \right| >N\sqrt d$ if
\begin{equation}	
 		{r_1} <w+
 		\begin{cases}
 			{{\ln {{N\sqrt d } \over t}} \over {\ln {\beta _{\max }}}},t \le N\sqrt d \;when\; w>r_1 \\
 			{{\ln {{N\sqrt d } \over t}} \over {\ln {\beta _{\min }}}},t > N\sqrt d\; when\;w\le r_1 .
 		\end{cases}
 		\label{eq_h_s}
 	\end{equation}
Conversely, if ${\lambda_{{w}} }(\mathop x\limits_ -) (y - x)  \in B(0,N\sqrt d)$, we have
\begin{equation}	
 		{r_1} \ge w+
 		\begin{cases}
 			{{\ln {{N\sqrt d } \over t}} \over {\ln {\beta _{\max }}}},t \le N\sqrt d \;when\; w>r_1 \\
 			{{\ln {{N\sqrt d } \over t}} \over {\ln {\beta _{\min }}}},t > N\sqrt d\; when\;w\le r_1 .
 		\end{cases}
 		\label{eq_h_s}
 	\end{equation}
Let $y = [x],x \in {\mathbb{R}}$ be least integer function.
Let
\begin{equation}	
 		r(N) = 
 		\begin{cases}
 			[{{\ln {{N\sqrt d } \over t}} \over {\ln {\beta _{\max }} }}],t \le N\sqrt d \\
 			[{{\ln {{N\sqrt d } \over t}} \over {\ln {\beta _{\min }} }}],t > N\sqrt d .
 		\end{cases}
 		\label{eq_h_s}
 	\end{equation}
As a result,
if ${\lambda_{{w}} }(\mathop x\limits_ -) (y - x)  \in B(0,N\sqrt d)$, we have 
$r_1 \ge w+r(N)$. That is,
\begin{center}
  ${y_r} = {x_r}$ for any $r < w+r(N)$.
\end{center}
This completes the proof. \\
(3). 
Indeed, there exists a positive integer $ N $ such that $ - \alpha  N <  {a_i}< {b_i} <  \alpha N, i = 1,2, \ldots ,d $. For the $N$, by Lemma \ref{le:3.61} (2), there exists an integer $r(N)$ such that
\begin{equation}
({\lambda _{{n_{{k}}}}}(\mathop x\limits_ -){C^x}) \cap \mathop \Pi \limits_{i = 1}^d[ - N,N] = (\lambda _{{n_{{k}}}}^{{n_{{k}}} +r(N)}(\mathop x\limits_ -){h^{n_k +r(N)}}(\mathop x\limits_ -)({C^{x_{{n_{{k}}} +r(N)}^ * }})) \cap \mathop \Pi \limits_{i = 1}^d[ - N,N],
\end{equation}
 where $x_{{n_{{k}}} +r(N)}^* = {S^{{n_{{k}}} +r(N)}}(x)$.\\
Hence
\begin{equation}\label{E:3.06}
    (\alpha {\lambda _{{n_{{k}}}}}(\mathop x\limits_ -){C^x}) \cap \mathop \Pi \limits_{i = 1}^d[{a_i},{b_i}] = (\alpha \lambda _{{n_{{k}}}}^{{n_{{k}}} +r(N)}(\mathop x\limits_ -){h^{n_k +r(N)}}(\mathop x\limits_ -)({C^{x_{{n_{{k}}} + r(N)}^ * }}))  \cap \mathop \Pi \limits_{i = 1}^d[{a_i},{b_i}].
\end{equation}
Let $\mathop \beta \limits^ -  = max\{ {({\beta _{\max }})^{r(N)}},{({\beta _{\min }})^{r(N)}}\}.$\\
Since
$\{ x_{{n_{{k}}} +r(N)}^*\}  \subset C$,
  $\{ \lambda _{{n_{{k}}}}^{{n_{{k}}} + r(N)}(\mathop x\limits_ -)\left| {k \in  {\mathbb{N}_+}} \right.\}\subset [0,\mathop \beta \limits^ -  ] $ has only a finite number of possible values and $\{{h^{n_k +r(N)}}(\mathop x\limits_ -)\}\subset O(d)$,
 we infer from
 Lemma \ref{le:3.61} (1) that there is a closed set ${F^ * }$, a subsequence $\{ {n_{k_l}}\} $ of $\{ {n_k}\} $, a positive number $\beta (N) \in (0,\mathop \beta \limits^ -  ]$, $h(N) \in O(d)$ and ${z_N} \in C$
  such that
  $$\rho ( (\alpha {\lambda _{{n_{{k_l}}}}(\mathop x\limits_ -)}{C^x}) \cap \mathop \Pi \limits_{i = 1}^d [{a_i},{b_i}],{F^ * })\to 0,$$
  $$F \cap \mathop \Pi \limits_{i = 1}^d ({a_i},{b_i})={F^ * } \cap \mathop \Pi \limits_{i = 1}^d ({a_i},{b_i}),$$
  $$\lambda _{{n_{{{k_l}}}}}^{{n_{{{k_l}}}} +r(N)}(\mathop x\limits_ -)= \beta (N),$$
  $$\tilde d({h^{n_{k_l} +r(N)}}(\mathop x\limits_ -),h(N)) \to 0$$
  and
  $$x_{{n_{{{k_l}}}} + r(N)}^* \to {z_N}.$$
 By \eqref{E:3.06} and Proposition \ref{pr:5} (3), then we have
  $$(\alpha {\lambda _{{n_{{k_l}}}}(\mathop x\limits_ -)}{C^x}) \cap \mathop \Pi \limits_{i = 1}^d [{a_i},{b_i}]=(\alpha\beta (N){h^{n_{k_l} +r(N)}}(\mathop x\limits_ -)({C^{x_{{n_{{k_l}}} +r(N)}^ * }}))\cap \mathop \Pi \limits_{i = 1}^d [{a_i},{b_i}],$$
  $$\rho((\alpha\beta (N){h^{n_{k_l} +r(N)}}(\mathop x\limits_ -)({C^{x_{{n_{{k_l}}} +r(N)}^ * }}))\cap \mathop \Pi \limits_{i = 1}^d [{a_i},{b_i}],{F^ * })\to 0$$
  and
  $$\alpha\beta (N){h^{n_{k_l} +r(N)}}(\mathop x\limits_ -)({C^{x_{{n_{{k_l}}} + r(N)}^ * }})\mathop  \to \limits^{\rho}\alpha\beta (N)h(N)({C^{z_N}}).$$
  By Lemma \ref{Le:3.5}, we have
$$(\alpha\beta (N)h(N)({C^{z_N}})) \cap \mathop \Pi \limits_{i = 1}^d ({a_i},{b_i})= {F^*} \cap \mathop \Pi \limits_{i = 1}^d ({a_i},{b_i}).$$
Hence 
$$F \cap \mathop \Pi \limits_{i = 1}^d ({a_i},{b_i})=(\alpha\beta (N)h(N)({C^{z_N}})) \cap \mathop \Pi \limits_{i = 1}^d ({a_i},{b_i}).$$
Hence we finish the proof of Lemma \ref{le:3.61} (3).
\qed
\end{prof}

\begin{lemma}\label{le:3.7}
Let $C\subset{{\mathbb{R}^d}}$ be a self-similar set satisfying the open set condition with $s = {\dim _H}(C)$. Suppose $C$ is not included in any hyperplane in ${{\mathbb{R}}^d}$. Let $M(d, d-1) $ be a set of all hyperplanes on ${\mathbb {R}^ d}$ and ${V_\varepsilon }(P) = \{ x \in {\mathbb{R}^d}\left| {d(x,P)} \right. \leqslant \varepsilon \} $ for a hyperplane $P\in M(d, d-1)$. Then\\
(1). ${{\cal H}^s}(C \cap P) = 0$ for any $P \in M(d,d-1)$.\\
(2). $\mathop {\lim }\limits_{\varepsilon  \to {\text{0}}} {{\cal H}^s}(C \cap {V_\varepsilon }(P))=0$ for any $P \in M(d,d-1)$.
\end{lemma}
\begin{prof}
(1).Mattila (\cite{ref15}; 4.3. Corollary) gave a proof of Lemma \ref{le:3.7} (1).\\
 (2). By (\cite{ref18}; 1.8 Theorem), we know that if $(X, \mathcal{M}, \mu)$ is a measure space,  $\left\{E_j\right\}_{j=1}^{\infty} \subset \mathcal{M}$, $ E_1 \supset E_2 \supset \cdots$ and $\mu\left(E_1\right)<\infty$, then $\mu \left( {\mathop  \cap \limits_{j = 1}^\infty  {E_j}} \right) = \mathop {\lim }\limits_{j \to \infty } \mu \left( {{E_j}} \right)$.
     Hence by Lemma \ref{le:3.7} (1) and 
     ${{\cal H}^s}(C)<\infty$, we have
     $\mathop {\lim }\limits_{\varepsilon  \to {\text{0}}} {{\cal H}^s}(C \cap {V_\varepsilon }(P))=0$ for any $P \in M(d,d-1)$.
     \qed
\end{prof}

\begin{definition}\label{de:2.}
For any $d$-dimensional cube $D\subset{{\mathbb{R}}^d}$, $D$ has exactly $2^d$ vertices and the boundary of $D$ donoted as $\partial D$ is composed of $2d$ $(d-1)$-dimensional cubes. Thus there exist $2d$ distinct hyperplanes denoted as $\{ {P_i} \in M(d,d - 1)|i = 1,2, \cdots ,2d\} $ satisfying $\partial D = (\mathop  \cup \limits_{i = 1}^{2d} {P_i}) \cap D$. Here ${P_i} \cap D$ denotes a $(d-1)$-dimensional cube for any $i = 1,2, \cdots ,2d$. Let $\{ {x_i} \in {{\mathbb{R}}^d}|i = 1,2, \cdots ,{2^d}\} $ be the vertex set of cube $D$, we have $D=conv(\{ {x_i} \in {{\mathbb{R}}^d}|i = 1,2, \cdots ,{2^d}\})$ and $\{ {x_i} \in {{\mathbb{R}}^d}|i = 1,2, \cdots ,{2^d}\}\subset \partial D $. For any $i = 1,2, \cdots ,2d$,
there exists $T_i \subset \{ {x_i} \in {{\mathbb{R}}^d}|i = 1,2, \cdots ,{2^d}\}$ with $card(T_i)=2^{d-1}$ such that $conv(T_i)={P_i} \cap D$, where $card(T_i)$ denotes the number of elements in $T_i$.
\end{definition}

\begin{lemma}\label{le:3.}
Let $C\subset{{\mathbb{R}^d}}$ be a self-similar set satisfying the open set condition with $s = {\dim _H}(C)$
and $D,{D_1},{D _2}, \cdots $ be $d$-dimensional cubes on ${{\mathbb{R}^d}}$. Suppose $C$ is not included in any hyperplane in ${{\mathbb{R}}^d}$. For any $k\in {{\mathbb{N}}_+}$, consider the vertex sets of $D$ and ${D _k}$ as $\{ {x_i} \in {{\mathbb{R}}^d}|i = 1,2, \cdots ,{2^d}\} $ and $\{ {x_{k,i}} \in {{\mathbb{R}}^d}|i = 1,2, \cdots ,{2^d}\} $, respectively. If $\left| {{x_i} - {x_{k,i}}} \right| \to 0$ as $k \to \infty $ for $i = 1,2, \cdots ,{2^d}$, then for any $y\in C$, any $g\in O(d)$, we have
$${{\cal H}^s}(g({C^y}) \cap D) = \mathop {\lim }\limits_{k \to \infty } {{\cal H}^s}(g({C^y}) \cap {D_k}).$$
\end{lemma}

\begin{prof}
Since $\left|x_i-x_{k, i}\right| \rightarrow 0$ as $k \rightarrow \infty$ for $i=1,2, \ldots, 2^d$, for any $\varepsilon>0$, there exists a positive number $K$ such that $\left|x_i-x_{k, i}\right|<\varepsilon$ for any $k \ge K$ and all $i=1,2, \cdots, 2^d$.\\
For any $k \in \mathbb{N}_{+}$, let $$H=\{P_i \in M(d, d-1) \mid i=1,2, \cdots, 2 d, \partial D=(\mathop  \cup \limits_{i = 1}^{2d} {P_i}) \cap D\}$$
and
$${H_k}=\{P_{k,i} \in M(d, d-1) \mid i=1,2, \ldots, 2 d, \partial D_k=(\mathop \cup \limits_{i = 1}^{2d} {P_{k,i}}) \cap D_k\}.$$
For any $i = 1,2, \cdots ,2d$, ${P_i} \cap D$ denotes a $(d-1)$-dimensional cube, and
there exists $T_i \subset \{ {x_i} \in {{\mathbb{R}}^d}|i = 1,2, \cdots ,{2^d}\}$ with $card(T_i)=2^{d-1}$ such that $conv(T_i)={P_i} \cap D$. Without losing generality, for any $i = 1,2, \cdots ,2d$, there exists $T_{k,i} \subset \{ {x_{k,i}} \in {{\mathbb{R}}^d}|i = 1,2, \cdots ,{2^d}\}$ with $card(T_{k,i})=2^{d-1}$ such that $conv(T_{k,i})={P_{k,i}} \cap D_k$ and $$d_H(T_i,T_{k,i})<\varepsilon$$
for any $k \ge K$.\\
Hence 
${T_{k,i}} \subset {V_\varepsilon }(conv({T_i})) = {V_\varepsilon }({P_i} \cap D)$ for any $k \ge K$ and any $i = 1,2, \cdots ,2d$.\\
Then
${P_{k,i}} \cap D_k = conv({T_{k,i}}) \subset {V_\varepsilon }({P_i} \cap D)$ for any $k \ge K$ and any $i = 1,2, \cdots ,2d$.\\
So we have
$$\partial D_k =\mathop \cup \limits_{i = 1}^{2d} ({P_{k,i}} \cap D_k)\subset \mathop \cup \limits_{i = 1}^{2d}({V_\varepsilon }({P_i} \cap D))={V_\varepsilon }(\mathop \cup \limits_{i = 1}^{2d}({P_i}) \cap D)= V_{\varepsilon}(\partial D)$$
for any $k \ge K$.\\
Hence we have $D_k={conv}(\{x_{k,1}, \cdots, x_{k, 2^d}\})\subset V_{\varepsilon}(D)$. \\
Similarly, $D={conv}(\{x_1, \cdots, x_{2^d}\}) \subset V_{\varepsilon}\left(D_k\right)$.\\
Then $${D_k} \backslash D \subset V_{\varepsilon}(D) \backslash D \subset V_{\varepsilon}(\partial D) \subset V_{\varepsilon}(\mathop  \cup \limits_{i = 1}^{2d} {P_i}) =\mathop \cup \limits_{i = 1}^{2d} V_{\varepsilon}(p_i)$$
and
$$D\backslash D_k \subset V_{\varepsilon}(D_k) \backslash D_k \subset V_{\varepsilon}\left(\partial D_k\right) \subset V_{2 \varepsilon}(\partial D) \subset \mathop \cup \limits_{i = 1}^{2d} V_{2\varepsilon}(p_i).$$
Let $D \Delta D_k=({D_k} \backslash D) \cup(D\backslash D_k)$, hence 
$$D \Delta D_k\subset(\mathop \cup \limits_{i = 1}^{2d} V_{\varepsilon}(p_i)) \cup(\mathop \cup \limits_{i = 1}^{2d} V_{2\varepsilon}(p_i))=\mathop \cup \limits_{i = 1}^{2d} V_{2\varepsilon}(p_i).$$
By Lemma \ref{le:3.7} (2), we have
$$\mathop {\lim }\limits_{k \to \infty } {{\cal H}^s}(g({C^y}) \cap (D \Delta D_k)))=0.$$
Hence
$${{\cal H}^s}(g({C^y}) \cap D) = \mathop {\lim }\limits_{k \to \infty } {{\cal H}^s}(g({C^y}) \cap {D_k}).$$
\qed
\end{prof}

\begin{lemma}\label{le:4.}
Let $C\subset{{\mathbb{R}^d}}$ be a self-similar set satisfying the open set condition with $s = {\dim _H}(C)$. Suppose $C$ is not included in any hyperplane in ${{\mathbb{R}}^d}$. Let $g, g_k \in O(d)$ for any $k \in \mathbb{N}_{+}$ with $\tilde{d}\left(g_k, g\right) \rightarrow 0$ as $k \rightarrow \infty$ and $y, y_k \in C$ for any $k \in \mathbb{N}_{+}$ with $\mathop {\lim} \limits_{k \rightarrow \infty} y_k=y$. Then for any $a_i, b_i \in \mathbb{R}, a_i<b_i, ,i = 1,2, \ldots ,d$,
we have 
$${{\cal H}^s}(g({C^y})\cap \mathop \Pi \limits_{i = 1}^d [{a_i},{b_i}]) = \mathop {\lim }\limits_{k \to \infty } {{\cal H}^s}(g_k({C^{y_k}}) \cap \mathop \Pi \limits_{i = 1}^d [{a_i},{b_i}]).$$
\end{lemma}
\begin{prof}
For any $k \in \mathbb{N}_{+}$, suppose the vertex sets of $\mathop \Pi \limits_{i = 1}^d [{a_i},{b_i}]$ and $g \circ g_k^{-1} (\mathop \Pi \limits_{i = 1}^d [{a_i},{b_i}])+g(y_k-y)$ are $\left\{x_1, x_2, \cdots, x_{2^d}\right\}$ and $\left\{x_{k, 1}, x_{k, 2}, \cdots, x_{k, 2^d}\right\}$ respectively, and $$x_{k, i}=g \circ g_k^{-1}\left(x_i\right)+g\left(y_k-y\right)$$ 
for $i=1,2, \ldots, 2^d$.\\
Since $\tilde{d}\left(g_k, g\right) \rightarrow 0$ as $k \rightarrow \infty$, 
$\tilde{d}\left(g \circ g_k^{-1}, I_d\right) \rightarrow 0$ as $k \rightarrow \infty$, where $I_d$ is an identity transformation.\\
For any $\varepsilon>0$, there exists a positive number $K$ such that $$\left|g \circ g_k^{-1}\left(x_i\right)-x_i\right|<\frac{\varepsilon}{2}$$
and
$$
\left|g\left(y_k-y\right)\right|=\left|y_k-y\right|<\frac{\varepsilon}{2}
$$
for $k \geqslant K$ and all $i=1,2, \cdots, 2^d$.
We have 
$$
\begin{aligned}
\left|x_{k, i}-x_i\right|&=\left|g \circ g_k^{-1}\left(x_i\right)+g\left(y_k-y\right)-x_i\right|\\& \leq\left|g\circ g_k^{-1}\left(x_i\right)-x_i\right|+\left|g\left(y_k-y\right)\right| \\
& <\frac{\varepsilon}{2}+\frac{\varepsilon}{2} \\
& =\varepsilon 
\end{aligned}
$$
for $k \geqslant K$ and all $i=1,2, \cdots, 2^d$.\\
Hence $\left|x_{k, i}-x_i\right| \rightarrow 0$ as $k \rightarrow \infty$ for $i=1,2, \cdots, 2^d$.\\
By Lemma \ref{le:3.}, we have
$${{\cal H}^s}(g({C^y}) \cap \mathop \Pi \limits_{i = 1}^d [{a_i},{b_i}]) = \mathop {\lim }\limits_{k \to \infty } {{\cal H}^s}(g({C^y}) \cap (g \circ g_k^{-1} (\mathop \Pi \limits_{i = 1}^d [{a_i},{b_i}])+g(y_k-y))).$$
That is,
$${{\cal H}^s}(g({C^y})\cap \mathop \Pi \limits_{i = 1}^d [{a_i},{b_i}]) = \mathop {\lim }\limits_{k \to \infty } {{\cal H}^s}(g_k({C^{y_k}}) \cap \mathop \Pi \limits_{i = 1}^d [{a_i},{b_i}]).$$
\qed
\end{prof}

\begin{lemma}\label{le:5.}
Let $C\subset{{\mathbb{R}^d}}$ be a self-similar set satisfying the open set condition with $s = {\dim _H}(C)$. Suppose $C$ is not included in any hyperplane in ${{\mathbb{R}}^d}$.  
Suppose $\left\{g_k\right\}\subset O(d)$ is a convergent sequence  by $\widetilde{d}$. Suppose $\left\{y_k\right\} \subset C$ and $\left\{u_k\right\} \subset \mathbb{R}$ converge with $\mathop {\lim }\limits_{k \rightarrow 0} u_k=1$. Then for any $a_i, b_i \in \mathbb{R}, a_i<b_i,i = 1,2, \ldots ,d$,
we have
$$\lim _{k \rightarrow \infty} {{\cal H}^s}(g_k (C^{y_k}) \cap \mathop \Pi \limits_{i = 1}^d [{a_i},{b_i}])=\lim _{k \rightarrow \infty} {{\cal H}^s}(g_k (C^{y_k}) \cap \mathop \Pi \limits_{i = 1}^d [{{a_i} \over {{u_k}}},{{b_i} \over {{u_k}}}]) .$$
\end{lemma}

\begin{prof}
Since $\left\{g_k\right\}\subset O(d)$ converges by $\widetilde{d}$ and $ O(d)$ is compact, there exists $g\in O(d)$ such that $\tilde{d}\left(g_k, g\right) \rightarrow 0$ as $k \rightarrow \infty$. Hence
$\tilde{d}\left(g \circ g_k^{-1}, I_d\right) \rightarrow 0$ as $k \rightarrow \infty$.\\
Since $\left\{y_k\right\} \subset C$ converges and $C$ is compact, then there exists $y\in C$ such that $\mathop {\lim }\limits_{k \rightarrow 0} y_k=y$.\\
For any $k \in \mathbb{N}_{+}$, suppose the vertex sets of $\mathop \Pi \limits_{i = 1}^d [{a_i},{b_i}]$ and $g \circ g_k^{-1} (\mathop \Pi \limits_{i = 1}^d [{{a_i} \over {{u_k}}},{{b_i} \over {{u_k}}}])+g(y_k-y)$ are $\left\{x_1, x_2, \cdots, x_{2^d}\right\}$ and $\left\{x_{k, 1}, x_{k, 2}, \cdots, x_{k, 2^d}\right\}$ respectively,
and $$x_{k, i}=g \circ g_k^{-1}({{x_i} \over {{u_k}}})+g(y_k-y)$$ 
for $i=1,2, \ldots, 2^d$.\\
    There exists a positive number $M$ such that
$$
\mathop \Pi \limits_{i = 1}^d [{a_i},{b_i}]\subset B(0, M).
$$
For any $\varepsilon \in (0,1/2)$, there exists a positive number $K$ such that $$\left|g \circ g_k^{-1}\left(x_i\right)-x_i\right|<{\varepsilon},$$
$$\left|1-u_k\right|<{\varepsilon}$$
and
$$
\left|g\left(y_k-y\right)\right|=\left|y_k-y\right|<{\varepsilon}
$$
for any $k \geqslant K$ and all $i=1,2, \cdots, 2^d$.
We have 
$$
\begin{aligned}
\left|x_{k, i}-x_i\right|&=\left|g \circ g_k^{-1}({{x_i} \over {{u_k}}})+g\left(y_k-y\right)-x_i\right|\\& \leq\left|{1 \over {{u_k}}}g\circ g_k^{-1}\left(x_i\right)-x_i\right|+\left|g\left(y_k-y\right)\right| \\
&\le\left|{1 \over {{u_k}}}g\circ g_k^{-1}\left(x_i\right)-{1 \over {{u_k}}}x_i\right|+\left|({1 \over {{u_k}}}-1)x_i\right|+\left|g\left(y_k-y\right)\right| \\
& <{{\varepsilon} \over {1 - \varepsilon }}+({1 \over {1 - \varepsilon }}-1)M+{\varepsilon} \\
& ={{\varepsilon} \over {1 - \varepsilon }}(1+M)+{\varepsilon} \\
& <(2M+3)\varepsilon 
\end{aligned}
$$
for any $k \geqslant K$ and all $i=1,2, \cdots, 2^d$.\\
Hence $\left|x_{k, i}-x_i\right| \rightarrow 0$ as $k \rightarrow \infty$ for $i=1,2, \cdots, 2^d$.\\
By Lemma \ref{le:3.}, we have
$${{\cal H}^s}(g({C^y}) \cap \mathop \Pi \limits_{i = 1}^d [{a_i},{b_i}]) = \mathop {\lim }\limits_{k \to \infty } {{\cal H}^s}(g({C^y}) \cap (g \circ g_k^{-1} (\mathop \Pi \limits_{i = 1}^d [{{a_i} \over {{u_k}}},{{b_i} \over {{u_k}}}])+g(y_k-y))).$$
That is,
$${{\cal H}^s}(g({C^y} )\cap \mathop \Pi \limits_{i = 1}^d [{a_i},{b_i}]) = \lim _{k \rightarrow \infty} {{\cal H}^s}(g_k( C^{y_k}) \cap \mathop \Pi \limits_{i = 1}^d [{{a_i} \over {{u_k}}},{{b_i} \over {{u_k}}}]).$$
Hence by Lemma \ref{le:4.}, we have
$$\lim _{k \rightarrow \infty} {{\cal H}^s}(g_k( C^{y_k}) \cap \mathop \Pi \limits_{i = 1}^d [{a_i},{b_i}])=\lim _{k \rightarrow \infty} {{\cal H}^s}(g_k( C^{y_k}) \cap \mathop \Pi \limits_{i = 1}^d [{{a_i} \over {{u_k}}},{{b_i} \over {{u_k}}}]) .$$
\qed
\end{prof}

\begin{proposition}\label{pr:5.2}
Let $C\subset{{\mathbb{R}^d}}$ be a self-similar set generated by an IFS $\{ {\varphi _i}= {\beta_i}g_i (x) + {c_i}\} _{i = 0}^{m-1}$, where $0 < {\beta_i}  < 1$, $g_i\in O(d)$ and ${c_i}\in {{\mathbb{R}^d}}$ for $i = 0,1, \cdots ,m - 1$. Suppose $\{ {\varphi _i}\} _{i = 0}^{m-1}$ satisfies the strong separation condition and $C$ is not included in any hyperplane in ${{\mathbb{R}}^d}$. Let $s = {\dim _H}(C)$ and $ \mathop \Pi \limits_{i = 1}^d [{a_i},{b_i}) = [{a_1},{b_1}) \times  \cdots  \times [{a_d},{b_d})$.
Then for any $x \in C$, any $F= (\rho )\mathop {\lim }\limits_{k \to \infty } \alpha {{{\lambda _{{n_k}}}(\mathop x\limits_ -)}}{C^x} \in \omega({C^x})$ with $x=\pi (\mathop x\limits_ -  )$, and any $ {a_i},{b_i} \in {\mathbb{R}}$, ${a_i} < {b_i}$, $i = 1,2, \ldots ,d$,
there exists a subsequence $\{n_{k_l}\} $ of $\{n_k\} $ such that
  \[{{\cal H}^s}(F \cap \mathop \Pi \limits_{i = 1}^d [\mathop {a_i}\limits^ - ,\mathop {b_i}\limits^ -)) = \mathop {\lim }\limits_{l \to \infty } {{\cal H}^s}((\alpha {{{\lambda  _{n_{k_l}}}(\mathop x\limits_ -)}}{C^x}) \cap \mathop \Pi \limits_{i = 1}^d [\mathop {a_i}\limits^ - ,\mathop {b_i}\limits^ -))\]
for any $\mathop {a_i}\limits^ - ,\mathop {b_i}\limits^-$ $\in \mathbb{R}$, $\mathop {a_i}\limits^ - <\mathop {b_i}\limits^-$ satisfying $\mathop \Pi \limits_{i = 1}^d [\mathop {a_i}\limits^ - ,\mathop {b_i}\limits^ -) \subset \mathop \Pi \limits_{i = 1}^d [{a_i},{b_i})$.
\end{proposition}
\begin{prof}
  The proof of Proposition \ref{pr:5.2} comprises two main steps. Firstly, to complete Proposition \ref{pr:5.2}, we demonstrate that our focus narrows to validating \eqref{E:5.13} by establishing the validity of two essential claims. Next, we furnish the comprehensive proof of \eqref{E:5.13}.\\
 \textbf{Claim 1. } $ {{\cal H}^s}( (\alpha{{{\lambda _{n_{k_l}}}(\mathop x\limits_ -)}}{C^x}) \cap \mathop \Pi \limits_{i = 1}^d [\mathop {a_i}\limits^ - ,\mathop {b_i}\limits^ -) )={{\cal H}^s}((\alpha{\beta (N)} {h^{n_{k_l} +r(N)}}(\mathop x\limits_ -)({{C^{x_{n_{k_l}+r(N)}^*}}}))\cap \mathop \Pi \limits_{i = 1}^d [\mathop {a_i}\limits^ - ,\mathop {b_i}\limits^ -))$ 
 for any $\mathop {a_i}\limits^ - ,\mathop {b_i}\limits^-$ $\in \mathbb{R}$, $\mathop {a_i}\limits^ - <\mathop {b_i}\limits^-$ satisfying $\mathop \Pi \limits_{i = 1}^d [\mathop {a_i}\limits^ - ,\mathop {b_i}\limits^ -) \subset \mathop \Pi \limits_{i = 1}^d [{a_i},{b_i})$.\\
 Indeed, by Lemma \ref{le:3.61} (3), for any $x \in C$, any $F= (\rho )\mathop {\lim }\limits_{k \to \infty } \alpha {{{\lambda _{{n_k}}}(\mathop x\limits_ -)}}{C^x} \in \omega({C^x})$  with $x=\pi (\mathop x\limits_ -  )$, and any $ {a_i},{b_i} \in {\mathbb{R}}$, ${a_i} < {b_i}$, $i = 1,2, \ldots ,d$, 
there exists a positive integer $N$ satisfying $\mathop \Pi \limits_{i = 1}^d [\frac{{a_i}}{\alpha },\frac{{b_i}}{\alpha }]\subset\mathop \Pi \limits_{i = 1}^d [-N,N]$, a subsequence  $\{n_{k_l}\} $ of $\{n_k\} $, $r(N)\in \mathbb{Z}$, $h(N) \in O(d)$, ${z_N} \in C$ and ${\beta(N)}>0$ such that
  \begin{equation}\label{E:5.7}
  (\alpha{{{\lambda _{n_{k_l}}}(\mathop x\limits_ -)}}{C^x}) \cap \mathop \Pi \limits_{i = 1}^d [{a_i},{b_i}] =(\alpha{\beta (N)}{h^{n_{k_l} +r(N)}}(\mathop x\limits_ -) ({{C^{x_{n_{k_l}+r(N)}^*}}}))\cap \mathop \Pi \limits_{i = 1}^d [{a_i},{b_i}] ,
  \end{equation}
  \begin{equation}\label{E:5.8}
  x_{n_{k_l}+r(N)}^* \to {z_N} ,
  \end{equation}
  \begin{equation}\label{E:5.08}
   \tilde d({h^{n_{k_l} +r(N)}}(\mathop x\limits_ -),h(N)) \to 0
  \end{equation}
  and
   \begin{equation}\label{E:5.9}
  F \cap \mathop \Pi \limits_{i = 1}^d 
({a_i},{b_i})= (\alpha {\beta (N)}h(N)({C^{{z_N}}})) \cap \mathop \Pi \limits_{i = 1}^d 
({a_i},{b_i}).
  \end{equation}
    For any $\mathop {a_i}\limits^ - ,\mathop {b_i}\limits^-$ $\in \mathbb{R}$, $\mathop {a_i}\limits^ - <\mathop {b_i}\limits^-$ satisfying $\mathop \Pi \limits_{i = 1}^d [\mathop {a_i}\limits^ - ,\mathop {b_i}\limits^ -) \subset \mathop \Pi \limits_{i = 1}^d [{a_i},{b_i})$, by \eqref{E:5.7} and \eqref{E:5.9}, we have
     \begin{equation}
  (\alpha{{{\lambda _{n_{k_l}}}(\mathop x\limits_ -)}}{C^x}) \cap \mathop \Pi \limits_{i = 1}^d [\mathop {a_i}\limits^ - ,\mathop {b_i}\limits^ -] =(\alpha{\beta (N)}{h^{n_{k_l} +r(N)}}(\mathop x\limits_ -) ({{C^{x_{n_{k_l}+r(N)}^*}}}))\cap \mathop \Pi \limits_{i = 1}^d [\mathop {a_i}\limits^ - ,\mathop {b_i}\limits^ -]
  \end{equation}
    and
    \begin{equation}\label{E:5.11}
     F \cap \mathop \Pi \limits_{i = 1}^d(\mathop {a_i}\limits^ - ,\mathop {b_i}\limits^ -)= (\alpha {\beta (N)}h(N)({C^{{z_N}}})) \cap \mathop \Pi \limits_{i = 1}^d (\mathop {a_i}\limits^ - ,\mathop {b_i}\limits^ -).
    \end{equation}
By Lemma \ref{le:3.7} (1), we have ${{\cal H}^s}(C\cap \partial {\mathop \Pi \limits_{i = 1}^d [{p_i},{q_i}]}) = 0$ and ${{\cal H}^s}(g(C)\cap \partial {\mathop \Pi \limits_{i = 1}^d [{p_i},{q_i}]}) = 0$ for any $ {p_i},{q_i} \in {\mathbb{R}}$, ${p_i} < {q_i}$, $i = 1,2, \ldots ,d$ and any $g \in O(d)$.\\
Hence 
\begin{equation}\label{E:5.07}
 {{\cal H}^s}( (\alpha{{{\lambda _{n_{k_l}}}(\mathop x\limits_ -)}}{C^x}) \cap \mathop \Pi \limits_{i = 1}^d [\mathop {a_i}\limits^ - ,\mathop {b_i}\limits^ -) )={{\cal H}^s}((\alpha{\beta (N)}{h^{n_{k_l} +r(N)}}(\mathop x\limits_ -) ({{C^{x_{n_{k_l}+r(N)}^*}}}))\cap \mathop \Pi \limits_{i = 1}^d [\mathop {a_i}\limits^ - ,\mathop {b_i}\limits^ -)).
  \end{equation}
This completes the proof of Claim 1.\\
\textbf{Claim 2. } ${{\cal H}^s}({F} \cap \partial {\mathop \Pi \limits_{i = 1}^d [\mathop {a_i}\limits^ - ,\mathop {b_i}\limits^ -]})= 0$ and 
${{\cal H}^s}(F \cap \mathop \Pi \limits_{i = 1}^d 
[\mathop {a_i}\limits^ - ,\mathop {b_i}\limits^ -))={{\cal H}^s}((\alpha {\beta (N)}h(N)({C^{{z_N}}})) \cap \mathop \Pi \limits_{i = 1}^d 
[\mathop {a_i}\limits^ - ,\mathop {b_i}\limits^ -))$ for any $\mathop {a_i}\limits^ - ,\mathop {b_i}\limits^-$ $\in \mathbb{R}$, $\mathop {a_i}\limits^ - <\mathop {b_i}\limits^-$ satisfying $\mathop \Pi \limits_{i = 1}^d [\mathop {a_i}\limits^ - ,\mathop {b_i}\limits^ -) \subset \mathop \Pi \limits_{i = 1}^d [{a_i},{b_i})$.\\
   By Lemma \ref{le:3.61}, similarly, given $\varepsilon  > 0$, there exists ${{z_N}^ *} \in C$, ${{h(N)}^ *}\in O(d)$ and ${{\beta (N)}^ *}>0$ such that
  $$F \cap \mathop \Pi \limits_{i = 1}^d (\mathop {a_i}\limits^ - - \varepsilon ,\mathop {b_i}\limits^ - + \varepsilon ) = (\alpha {{\beta (N)}^ *}{{h(N)}^ *}({C^{{{z_N}^ *}}})) \cap \mathop \Pi \limits_{i = 1}^d (\mathop {a_i}\limits^ - - \varepsilon ,\mathop {b_i}\limits^ - + \varepsilon ).$$
  Given a sequence $\{{\varepsilon_n}\}\subset (0,\varepsilon)$ such that ${\varepsilon _n} \to 0$ as $n \to \infty $,
  we have
  $$F \cap {\mathop \Pi \limits_{i = 1}^d (\mathop {a_i}\limits^ - - {\varepsilon_n} ,\mathop {b_i}\limits^ - + {\varepsilon_n})} = (\alpha {{\beta (N)}^ *}{{h(N)}^ *}({C^{{{z_N}^ *}}})) \cap \mathop \Pi \limits_{i = 1}^d (\mathop {a_i}\limits^ - - {\varepsilon_n} ,\mathop {b_i}\limits^ - + {\varepsilon_n} )$$
  and
  $$F \cap {\mathop \Pi \limits_{i = 1}^d (\mathop {a_i}\limits^ - ,\mathop {b_i}\limits^ - )} = (\alpha {{\beta (N)}^ *}{{h(N)}^ *}({C^{{{z_N}^ *}}})) \cap \mathop \Pi \limits_{i = 1}^d (\mathop {a_i}\limits^ -,\mathop {b_i}\limits^ - ).$$
  Then we have
  \[
\begin{aligned}
&{{\cal H}^s}(F \cap \partial {\mathop \Pi \limits_{i = 1}^d [\mathop {a_i}\limits^ - ,\mathop {b_i}\limits^ -]})\\
\le& {{\cal H}^s}(F \cap ({\mathop \Pi \limits_{i = 1}^d (\mathop {a_i}\limits^ - - {\varepsilon_n} ,\mathop {b_i}\limits^ - + {\varepsilon_n} )}\backslash  {\mathop \Pi \limits_{i = 1}^d(\mathop {a_i}\limits^ - ,\mathop {b_i}\limits^ -)}))\\
  = &{{\cal H}^s}((\alpha {{\beta (N)}^ *}{{h(N)}^ *}({C^{{{z_N}^ *}}})) \cap ({\mathop \Pi \limits_{i = 1}^d (\mathop {a_i}\limits^ - - {\varepsilon_n} ,\mathop {b_i}\limits^ - + {\varepsilon_n} )}\backslash \mathop \Pi \limits_{i = 1}^d (\mathop {a_i}\limits^ - ,\mathop {b_i}\limits^ -)))
\end{aligned}
\]
  for any $n \in {{\mathbb{N}}_ + }.$\\
  By Lemma \ref{le:3.7} (2), we have
  \begin{center}
    ${{\cal H}^s}((\alpha {{\beta (N)}^ *}{{h(N)}^ *}({C^{{{z_N}^ *}}})) \cap ({\mathop \Pi \limits_{i = 1}^d (\mathop {a_i}\limits^ - - {\varepsilon_n} ,\mathop {b_i}\limits^ - + {\varepsilon_n} )}\backslash \mathop \Pi \limits_{i = 1}^d (\mathop {a_i}\limits^ - ,\mathop {b_i}\limits^ -))) \to 0$ as ${\varepsilon_n}  \to 0$.
  \end{center}
  Then
     ${{\cal H}^s}({F} \cap \partial {\mathop \Pi \limits_{i = 1}^d [\mathop {a_i}\limits^ - ,\mathop {b_i}\limits^ -]}) = 0$
    for any $\mathop {a_i}\limits^ - ,\mathop {b_i}\limits^-$ $\in \mathbb{R}$, $\mathop {a_i}\limits^ - <\mathop {b_i}\limits^-$ satisfying $\mathop \Pi \limits_{i = 1}^d [\mathop {a_i}\limits^ - ,\mathop {b_i}\limits^ -) \subset \mathop \Pi \limits_{i = 1}^d [{a_i},{b_i})$.\\
By Lemma \ref{le:3.7} and \eqref{E:5.11}, then we have 
\begin{equation}\label{E:5.011}
{{\cal H}^s}(F \cap \mathop \Pi \limits_{i = 1}^d 
[\mathop {a_i}\limits^ - ,\mathop {b_i}\limits^ -))={{\cal H}^s}((\alpha {\beta (N)}{h(N)}({C^{{z_N}}})) \cap \mathop \Pi \limits_{i = 1}^d 
[\mathop {a_i}\limits^ - ,\mathop {b_i}\limits^ -))
\end{equation}
This completes the proof of Claim 2.\\
By Claim 1 and Claim 2, we just need to prove
 \begin{center}
   ${{\cal H}^s}((\alpha{\beta(N)}{h(N)}({C^{z_N}})) \cap  \mathop \Pi \limits_{i = 1}^d[{\mathop {a_i}\limits^ - },{\mathop {b_i}\limits^ - })) = \mathop {\lim }\limits_{l \to \infty } {{\cal H}^s}((\alpha{\beta(N)}{h^{n_{k_l} +r(N)}}(\mathop x\limits_ -)({C^{x_{{n_{k_l}} + r(N)}^ * }})) \cap \mathop \Pi \limits_{i = 1}^d[ {\mathop {a_i}\limits^ - },{\mathop {b_i}\limits^ - })).$
 \end{center}
  For convenience, replace $\frac{\mathop {a_i}\limits^ - }{\alpha \beta (N)}$, $\frac{\mathop {b_i}\limits^ - }{\alpha \beta (N)}$ with ${l_{i,1}}$,${l_{i,2}}$. \\
  We need to prove
\begin{equation}\label{E:5.13}
    \begin{aligned}
{{\cal H}^s}(h(N)({C^{z_N}}) \cap \mathop \Pi \limits_{i = 1}^d[{l_{i,1}},{l_{i,2}})) = \mathop {\lim }\limits_{l \to \infty } {{\cal H}^s}({h^{n_{k_l} +r(N)}}(\mathop x\limits_ -)({C^{x_{{n_{k_l}} +r(N)}^ * }}) \cap \mathop \Pi \limits_{i = 1}^d[{l_{i,1}},{l_{i,2}})).
  \end{aligned}
\end{equation}
Finally, we shall give the proof of \eqref{E:5.13} below.\\
 In fact, because of 
 \eqref{E:5.8} and \eqref{E:5.08}, by Lemma \ref{le:4.},
 we have
 $${{\cal H}^s}(h(N)({C^{z_N}}) \cap \mathop \Pi \limits_{i = 1}^d[{l_{i,1}},{l_{i,2}})) = \mathop {\lim }\limits_{l \to \infty } {{\cal H}^s}({h^{n_{k_l} +r(N)}}(\mathop x\limits_ -)({C^{x_{{n_{k_l}} +r(N)}^ * }}) \cap \mathop \Pi \limits_{i = 1}^d[{l_{i,1}},{l_{i,2}})).$$
 This completes the proof. \qed
\end{prof}
\begin{corollary}\label{co:3.9}
 Let $C$, $\{ {\varphi _i}\} _{i = 0}^{m-1}$ and $s$ be as in Proposition \ref{pr:5.2}.
  Then for any $x \in C$, any $ F = (\rho )\mathop {\lim }\limits_{k \to \infty } \ {e^{{t_k}}}{C^x} \in \omega(C^x)$ and any $ {a_i},{b_i} \in {\mathbb{R}}$, ${a_i} < {b_i}$, $i = 1,2, \ldots ,d$, there exists a subsequence $\{ {e^{{t_{{s_k}}}}}\}  \subset \{ {e^{{t_k}}}\} $ such that
  \begin{equation}\label{E:3.6}
{{\cal H}^s}(F \cap \mathop \Pi \limits_{i = 1}^d [\mathop {a_i}\limits^ - ,\mathop {b_i}\limits^ -)) = \mathop {\lim }\limits_{k \to \infty } {{\cal H}^s}(({e^{{t_{{s_k}}}}}{C^x}) \cap \mathop \Pi \limits_{i = 1}^d [\mathop {a_i}\limits^ - ,\mathop {b_i}\limits^ -))
\end{equation}
for any $\mathop {a_i}\limits^ - ,\mathop {b_i}\limits^-$ $\in \mathbb{R}$, $\mathop {a_i}\limits^ - <\mathop {b_i}\limits^-$ satisfying $\mathop \Pi \limits_{i = 1}^d [\mathop {a_i}\limits^ - ,\mathop {b_i}\limits^ -) \subset \mathop \Pi \limits_{i = 1}^d [{a_i},{b_i})$.
\end{corollary}
\begin{prof}
  The proof of Corollary \ref{co:3.9} can be established in two steps. Firstly, we only need to prove \eqref{E:3.11} to prove Corollary \ref{co:3.9}. We explain below.\\
  For any $ F =(\rho ) \mathop {\lim }\limits_{k \to \infty } \ {e^{{t_k}}}{C^x} \in \omega(C^x)$, by the proof of Proposition \ref{pr:3.3} and Corollary \ref{co:3.4}, there exists a positive number
  $ \alpha$ and ${n_k} \to \infty $ as $k \to \infty $ such that
  \begin{equation}\label{E:3.7}
  \alpha {{\lambda _{{n_k}}}(\mathop x\limits_ -)}{C^x}\mathop  \to \limits^\rho  F,
  \end{equation}
  where $x=\pi (\mathop x\limits_ -  )$.
  And there exists a subsequence $\{ {e^{{t_{{s_k}}}}}\}  \subset \{ {e^{{t_k}}}\} $ so that
  \begin{equation}\label{E:3.8}
  \mathop {\lim }\limits_{k \to \infty } \frac{{{e^{{t_{{s_k}}}}}}} {{\lambda _{{n_k}}}(\mathop x\limits_ -)} =\alpha  \in [1,{\beta _{\min }} ^{-1}],
  \end{equation}
 where ${\beta _{\min }} = \mathop {\min }\limits_{0 \le i \le m - 1} ({\beta _i})$.\\
  There exists a non-negative integer $ N $ such that $ - \alpha N < {a_i} < {b_i}< \alpha N$, $i = 1,2, \ldots ,d$.
 By Lemma \ref{le:3.61} (2)-(3) and Proposition \ref{pr:5.2}, there exists a sequence $\{ {n_k}\} $ satisfying \eqref{E:3.7}-\eqref{E:3.8}, a integer $r(N)$, ${\beta(N)}>0$, $h(N) \in O(d)$ and ${z_N} \in C$ so that
 \begin{equation}\label{E:3.17}
   \tilde d({h^{n_k +r(N)}}(\mathop x\limits_ -),h(N)) \to 0
  \end{equation}
 \begin{equation}\label{E:3.9}
 x_{{n_k}+r(N)}^* \to {z_N},
  \end{equation}
 \begin{equation}\label{E:3.90}
 (\alpha{{{\lambda _{n_k}}(\mathop x\limits_ -)}}{C^x}) \cap \mathop \Pi \limits_{i = 1}^d [-\alpha N,\alpha N] =(\alpha{\beta (N)} {h^{n_k +r(N)}}(\mathop x\limits_ -)({{C^{x_{n_k+r(N)}^*}}}))\cap \mathop \Pi \limits_{i = 1}^d [-\alpha N,\alpha N] 
  \end{equation}
 and
   \begin{equation}\label{E:3.10}
 {{\cal H}^s}(F \cap \mathop \Pi \limits_{i = 1}^d [\mathop {a_i}\limits^ - ,\mathop {b_i}\limits^ -) ) = \mathop {\lim }\limits_{k \to \infty } {{\cal H}^s}((\alpha {{\lambda _{{n_k}}}(\mathop x\limits_ -)}{C^x}) \cap \mathop \Pi \limits_{i = 1}^d [\mathop {a_i}\limits^ - ,\mathop {b_i}\limits^ -) )
  \end{equation}
  for any $\mathop {a_i}\limits^ - ,\mathop {b_i}\limits^-$ $\in \mathbb{R}$, $\mathop {a_i}\limits^ - <\mathop {b_i}\limits^-$ satisfying $\mathop \Pi \limits_{i = 1}^d [\mathop {a_i}\limits^ - ,\mathop {b_i}\limits^ -) \subset \mathop \Pi \limits_{i = 1}^d [{a_i},{b_i})$.\\
  In fact, we can choose $\{{n_k}\}$ and $\{{s_k}\}$ to simultaneously satisfy \eqref{E:3.7}-\eqref{E:3.10}.\\
  As a result, to prove \eqref{E:3.6}, we only need to prove
   \begin{equation}
  \mathop {\lim }\limits_{k \to \infty } {{\cal H}^s}((\alpha {{\lambda _{{n_k}}}(\mathop x\limits_ -)}{C^x}) \cap \mathop \Pi \limits_{i = 1}^d [\mathop {a_i}\limits^ - ,\mathop {b_i}\limits^ -))=\mathop {\lim }\limits_{k \to \infty } {{\cal H}^s}(( {e^{{{t_{{s_k}}}}}}{C^x}) \cap \mathop \Pi \limits_{i = 1}^d [\mathop {a_i}\limits^ - ,\mathop {b_i}\limits^ -)).
   \end{equation}
  Let ${u_k}=\frac{{{e^{{t_{{s_k}}}}}}}{{\alpha {{\lambda _{{n_k}}}(\mathop x\limits_ -)}}}$, then $\mathop {\lim }\limits_{k \to \infty } {u_k} = 1$.\\
  We have
  $${{\cal H}^s}(({e^{{t_{{s_k}}}}}{C^x}) \cap \mathop \Pi \limits_{i = 1}^d [\mathop {a_i}\limits^ - ,\mathop {b_i}\limits^ -))= {({u_k})^s}{{\cal H}^s}((\alpha{{\lambda _{{n_k}}}(\mathop x\limits_ -)}{C^x}) \cap \mathop \Pi \limits_{i = 1}^d [{{\mathop {a_i}\limits^ -} \over {{u_k}}},{{\mathop {b_i}\limits^ -} \over {{u_k}}}))$$
  and
  $$\mathop {\lim }\limits_{k \to \infty }{{\cal H}^s}(({e^{{t_{{s_k}}}}}{C^x}) \cap \mathop \Pi \limits_{i = 1}^d [\mathop {a_i}\limits^ - ,\mathop {b_i}\limits^ -))= \mathop {\lim }\limits_{k \to \infty }{{\cal H}^s}((\alpha{{\lambda _{{n_k}}}(\mathop x\limits_ -)}{C^x}) \cap \mathop \Pi \limits_{i = 1}^d [{{\mathop {a_i}\limits^ -} \over {{u_k}}},{{\mathop {b_i}\limits^ -} \over {{u_k}}})).$$
  Hence we need to prove
  \begin{equation}
  \mathop {\lim }\limits_{k \to \infty } {{\cal H}^s}((\alpha {{\lambda _{{n_k}}}(\mathop x\limits_ -)}{C^x}) \cap \mathop \Pi \limits_{i = 1}^d [\mathop {a_i}\limits^ - ,\mathop {b_i}\limits^ -))=\mathop {\lim }\limits_{k \to \infty } {{\cal H}^s}((\alpha{{\lambda _{{n_k}}}(\mathop x\limits_ -)}{C^x}) \cap \mathop \Pi \limits_{i = 1}^d [{{\mathop {a_i}\limits^ -} \over {{u_k}}},{{\mathop {b_i}\limits^ -} \over {{u_k}}})).
   \end{equation}
When $k$ is large enough, we have $\mathop \Pi \limits_{i = 1}^d [{{{\mathop {a_i}\limits^ -}} \over { {u_k}}},{{{\mathop {b_i}\limits^ -}} \over {{u_k}}}) \subset \mathop \Pi \limits_{i = 1}^d [-\alpha N,\alpha N)$.\\
 By \eqref{E:3.90}, then
 $$(\alpha{{{\lambda _{n_k}}(\mathop x\limits_ -)}}{C^x}) \cap \mathop \Pi \limits_{i = 1}^d [\mathop {a_i}\limits^ - ,\mathop {b_i}\limits^ -) =(\alpha{\beta (N)}{h^{n_k +r(N)}}(\mathop x\limits_ -) ({{C^{x_{n_k+r(N)}^*}}}))\cap \mathop \Pi \limits_{i = 1}^d [\mathop {a_i}\limits^ - ,\mathop {b_i}\limits^ -)$$
  and
 $$(\alpha{{{\lambda _{n_k}}(\mathop x\limits_ -)}}{C^x}) \cap \mathop \Pi \limits_{i = 1}^d [{{{\mathop {a_i}\limits^ -}} \over { {u_k}}},{{{\mathop {b_i}\limits^ -}} \over {{u_k}}}) =(\alpha{\beta (N)} {h^{n_k +r(N)}}(\mathop x\limits_ -)({{C^{x_{n_k+r(N)}^*}}}))\cap \mathop \Pi \limits_{i = 1}^d [{{{\mathop {a_i}\limits^ -}} \over { {u_k}}},{{{\mathop {b_i}\limits^ -}} \over {{u_k}}}).$$ 
 Hence we need to prove
 $$
 \mathop {\lim }\limits_{k \to \infty } {{\cal H}^s}((\alpha{\beta (N)} {h^{n_k +r(N)}}(\mathop x\limits_ -)({C^{x_{{n_k} + r(N)}^*}})) \cap \mathop \Pi \limits_{i = 1}^d [\mathop {a_i}\limits^ - ,\mathop {b_i}\limits^ -))$$
 $$=\mathop {\lim }\limits_{k \to \infty }{{\cal H}^s}( (\alpha{\beta (N)} {h^{n_k +r(N)}}(\mathop x\limits_ -)({C^{x_{{n_k} + r(N)}^*}}))\cap \mathop \Pi \limits_{i = 1}^d [{{{\mathop {a_i}\limits^ -}} \over { {u_k}}},{{{\mathop {b_i}\limits^ -}} \over {{u_k}}})).
 $$
For convenience, replace $\frac{\mathop {a_i}\limits^ - }{\alpha \beta (N)}$, $\frac{\mathop {b_i}\limits^ - }{\alpha \beta (N)}$ with ${l_{i,1}}$,${l_{i,2}}$.\\
  Then we just need to prove
 \begin{equation}\label{E:3.11}
\begin{aligned}
&\mathop {\lim }\limits_{k \to \infty } {{\cal H}^s}({h^{n_k +r(N)}}(\mathop x\limits_ -)({C^{x_{{n_k} +r(N)}^*}}) \cap \mathop \Pi \limits_{i = 1}^d[{l_{i,1}},{l_{i,2}}))\\
=&\mathop {\lim }\limits_{k \to \infty } {{\cal H}^s}({h^{n_k +r(N)}}(\mathop x\limits_ -)({C^{x_{{n_k} + r(N)}^*}}) \cap  \mathop \Pi \limits_{i = 1}^d[{{{l_{i,1}}} \over {{u_k}}},{{{l_{i,2}}} \over {{u_k}}})).
\end{aligned}
\end{equation}
 At this time, because of 
 \eqref{E:3.17} and \eqref{E:3.9}, by Lemma \ref{le:5.},
  then \eqref{E:3.11} holds.\\
  This completes the proof. \qed
\end{prof}

\begin{corollary}\label{co:3.11}
Let $C$, $\{ {\varphi _i}\} _{i = 0}^{m-1}$ and $s$ be as in Proposition \ref{pr:5.2}. 
 Then for any $x \in C$ and any $ F = (\rho )\mathop {\lim }\limits_{k \to \infty } \ {e^{{t_k}}}{C^x} \in \omega(C^x)$, there exists a subsequence $\{ {e^{{t_{{s_k}}}}}\}  \subset \{ {e^{{t_k}}}\} $ such that
  \begin{equation}
{{\cal H}^s}(F \cap \mathop \Pi \limits_{i = 1}^d [\mathop {a_i}\limits^ - ,\mathop {b_i}\limits^ -)) = \mathop {\lim }\limits_{k \to \infty } {{\cal H}^s}(({e^{{t_{{s_k}}}}}{C^x}) \cap \mathop \Pi \limits_{i = 1}^d [\mathop {a_i}\limits^ - ,\mathop {b_i}\limits^ -))
\end{equation}
for any $ \mathop {a_i}\limits^ - ,\mathop {b_i}\limits^ - \in {\mathbb{R}}$, $\mathop {a_i}\limits^ - <\mathop {b_i}\limits^ -$, $i = 1,2, \ldots ,d$.
\end{corollary}
\begin{prof}
For any $x \in C$ and any $ F = (\rho )\mathop {\lim }\limits_{k \to \infty } \ {e^{{t_k}}}{C^x} \in \omega(C^x)$, we replace $ \mathop \Pi \limits_{i = 1}^d [{a_i},{b_i})$ in Corollary \ref{co:3.9} with $\mathop \Pi \limits_{i = 1}^d [- j, j)$, where $
j\in {{\mathbb{N}}_+ }$.\\
(1)When $j=1$, by Corollary \ref{co:3.9}, there exists a monotonically increasing subsequence $\{e^{{t_{1,k}}}\} $ of $ \{ {e^{{t_k}}}\} $ such that
$$F =(\rho ) \mathop {\lim }\limits_{k \to \infty } {e^{{t_{1,k}}}}{C^x}$$
 and
 \begin{eqnarray}
  {{\cal H}^s}(F \cap \mathop \Pi \limits_{i = 1}^d [\mathop {a_i}\limits^ - ,\mathop {b_i}\limits^ -)) = \mathop {\lim }\limits_{k \to \infty } {{\cal H}^s}(({e^{{t_{1,k}}}}{C^x}) \cap \mathop \Pi \limits_{i = 1}^d [\mathop {a_i}\limits^ - ,\mathop {b_i}\limits^ -))
  \end{eqnarray}
for any $ \mathop {a_i}\limits^ - ,\mathop {b_i}\limits^ - \in {\mathbb{R}}$, $\mathop {a_i}\limits^ - <\mathop {b_i}\limits^ -$ satisfying $\mathop \Pi \limits_{i = 1}^d [\mathop {a_i}\limits^ - ,\mathop {b_i}\limits^ -) \subset \mathop \Pi \limits_{i = 1}^d[-1,1)$.\\
(2)When $j=2$, by Corollary \ref{co:3.9}, there exists a subsequence $\{e^{t_{{2,k}}}\}\subset \{e^{{t_{1,k}}}\} $
such that
$$F =(\rho ) \mathop {\lim }\limits_{k \to \infty } {e^{{t_{2,k}}}}{C^x}$$
 and
 \begin{eqnarray}
  {{\cal H}^s}(F \cap \mathop \Pi \limits_{i = 1}^d [\mathop {a_i}\limits^ - ,\mathop {b_i}\limits^ -)) = \mathop {\lim }\limits_{k \to \infty } {{\cal H}^s}(({e^{{t_{2,k}}}}{C^x}) \cap \mathop \Pi \limits_{i = 1}^d [\mathop {a_i}\limits^ - ,\mathop {b_i}\limits^ -))
  \end{eqnarray}
for any $ \mathop {a_i}\limits^ - ,\mathop {b_i}\limits^ - \in {\mathbb{R}}$, $\mathop {a_i}\limits^ - <\mathop {b_i}\limits^ -$ satisfying $\mathop \Pi \limits_{i = 1}^d [\mathop {a_i}\limits^ - ,\mathop {b_i}\limits^ -) \subset \mathop \Pi \limits_{i = 1}^d[-2,2)$.\\
   (3) When $j>2$, by Corollary \ref{co:3.9}, there exists a subsequence $\{e^{{t_{j,k}}}\}\subset \{e^{{t_{j-1,k}}}\} $
such that 
   $$F =(\rho ) \mathop {\lim }\limits_{k \to \infty } {e^{{t_{j,k}}}}{C^x}$$
 and
\begin{eqnarray}
  {{\cal H}^s}(F \cap \mathop \Pi \limits_{i = 1}^d [\mathop {a_i}\limits^ - ,\mathop {b_i}\limits^ -)) = \mathop {\lim }\limits_{k \to \infty } {{\cal H}^s}(({e^{{t_{j,k}}}}{C^x}) \cap \mathop \Pi \limits_{i = 1}^d [\mathop {a_i}\limits^ - ,\mathop {b_i}\limits^ -))
  \end{eqnarray}
for any $ \mathop {a_i}\limits^ - ,\mathop {b_i}\limits^ - \in {\mathbb{R}}$, $\mathop {a_i}\limits^ - <\mathop {b_i}\limits^ -$ satisfying $\mathop \Pi \limits_{i = 1}^d [\mathop {a_i}\limits^ - ,\mathop {b_i}\limits^ -) \subset \mathop \Pi \limits_{i = 1}^d[-j,j)$.\\
By the diagonal principle, we obtain a new sequence $\{e^{{t_{s_k}}}\} $ by setting ${e^{{t_{s_k}}}} ={e^{{t_{k,k}}}} $.\\
Since $\{ {e^{{t_{s_k}}}}\} _{k = j}^\infty $ is a subsequence of $\{e^{{t_{j,k}}}\}_{k = 1}^\infty$ for any $j \in {{\mathbb{N}}_+ }$,
hence
$$F =(\rho ) \mathop {\lim }\limits_{k \to \infty } {e^{{t_{s_k}}}}{C^x}$$
and
\begin{equation}
{{\cal H}^s}(F \cap \mathop \Pi \limits_{i = 1}^d [\mathop {a_i}\limits^ - ,\mathop {b_i}\limits^ -)) = \mathop {\lim }\limits_{k \to \infty } {{\cal H}^s}(( {e^{{t_{{s_k}}}}}{C^x}) \cap \mathop \Pi \limits_{i = 1}^d [\mathop {a_i}\limits^ - ,\mathop {b_i}\limits^ -))
\end{equation}
for any $ \mathop {a_i}\limits^ - ,\mathop {b_i}\limits^ - \in {\mathbb{R}}$, $\mathop {a_i}\limits^ - <\mathop {b_i}\limits^ -$, $i = 1,2, \ldots ,d$.\qed
\end{prof}
 
\section{Proof of Theorem 1.1}
 In this section, we prove some results about tangent measures on self-similar sets in ${{\mathbb{R}^d}}$ which satisfy the strong separation condition by Corollary \ref{co:3.11}. 
\begin{theorem}\label{th:4.1}
Let $C\subset{{\mathbb{R}^d}}$ be a self-similar set satisfying the strong separation condition with $s = {\dim _H}(C)$. Suppose $C$ is not included in any hyperplane in ${{\mathbb{R}}^d}$. Then for any $x \in C$ and any $ F = (\rho )\mathop {\lim }\limits_{k \to \infty } \ {e^{{t_k}}}{C^x} \in \omega(C^x)$, there exists a subsequence $\{ {e^{{t_{{s_k}}}}}\}  \subset \{ {e^{{t_k}}}\} $ such that
   $${{\cal H}^s}\lfloor _{{e^{{t_{s_k}}}}{C^x}}\mathop  \to \limits^w {{\cal H}^s}\lfloor _F.$$
\end{theorem}
\begin{prof}
By Corollary \ref{co:3.11}, then for any $x \in C$ and any $ F = (\rho )\mathop {\lim }\limits_{k \to \infty } \ {e^{{t_k}}}{C^x} \in \omega(C^x)$, there exists a subsequence $\{ {e^{{t_{{s_k}}}}}\}  \subset \{ {e^{{t_k}}}\} $ such that
  \begin{equation}\label{E:4.1}
{{\cal H}^s}(F \cap \mathop \Pi \limits_{i = 1}^d [\mathop {a_i}\limits^ - ,\mathop {b_i}\limits^ -)) = \mathop {\lim }\limits_{k \to \infty } {{\cal H}^s}(({e^{{t_{{s_k}}}}}{C^x}) \cap \mathop \Pi \limits_{i = 1}^d [\mathop {a_i}\limits^ - ,\mathop {b_i}\limits^ -))
\end{equation}
for any $ \mathop {a_i}\limits^ - ,\mathop {b_i}\limits^ - \in {\mathbb{R}}$, $\mathop {a_i}\limits^ - <\mathop {b_i}\limits^ -$, $i = 1,2, \ldots ,d$.\\
 We shall only prove
$${{\cal H}^s}\lfloor _{e^{{t_{s_k}}}{C^x}}\mathop  \to \limits^w {{\cal H}^s}\lfloor _F.$$
That is, we shall prove
\begin{eqnarray}\label{E:4.11}
  \mathop {\lim }\limits_{k \to \infty } \int {fd{\mu _k} = \int {fd{\mu _0}} } \; for \; all \; f \in {C_c}({{\mathbb{R}^d}}),
  \end{eqnarray}
  where ${\mu _k} = {{\cal H}^s}\lfloor_{e^{{t_{s_k}}}{C^x}}$, ${\mu _0} = {{\cal H}^s}\lfloor _F$,  and ${C_c}({{\mathbb{R}^d}})$ denotes the space of compactly supported continuous real-valued functions on ${{\mathbb{R}^d}}$.\\
   Fix any $f \in {C_c}({{\mathbb{R}^d}})$, there exists a positive integer $l$ so that $\operatorname{spt} f\subset \mathop \Pi \limits_{i = 1}^d [ - l,l)$, where $\operatorname{spt} f = \overline {\{ x \in {{\mathbb{R}^d}}:f(x) \ne 0\} }$.\\
   We only need to prove
  $$\mathop {\lim }\limits_{k \to \infty } \int_{\mathop \Pi \limits_{i = 1}^d [ - l,l)} {fd{\mu _{k}} = } \int_{\mathop \Pi \limits_{i = 1}^d [ - l,l)} {fd{\mu _0}.}$$
  First, if $\varphi $ is a step function, i.e. $\varphi=\sum\limits_{j = 1}^v {{u_j}{\chi _{ {D_j}}}} $ with $ {D_j}=\mathop \Pi \limits_{i = 1}^d[{x_{i,j}},{y_{i,j}})$, $\mathop {\mathop  \cup \limits_{j = 1} }\limits^v  {D_j} \subset \mathop \Pi \limits_{i = 1}^d [ - l,l)$ and $ {D_j} \cap  {D_k} = \emptyset $ for any $j \ne k \in \{ 1,2, \cdots ,v\} $,
  then $$\int {\varphi d{\mu _{k}}}  = \sum\limits_{j = 1}^v {{u_j}{\mu _{k}}( {D_j})} $$
 and
  $$\int {\varphi d{\mu _0}}  = \sum\limits_{j = 1}^v {{u_j}{\mu _0}({D_j})} .$$
  By \eqref{E:4.1}, we have
 \[{{\cal H}^s}(F \cap {D_j}) = \mathop {\lim }\limits_{k \to \infty } {{\cal H}^s}(({e^{{t_{s_k}}}}{C^x}) \cap  {D_j})\]
  for $j = 1,2, \cdots ,v$.\\
  That is, $${\mu _0}({D_j}) =\mathop {\lim }\limits_{k \to \infty }{\mu _{k}}( {D_j})$$ for $j = 1,2, \cdots ,v$.\\
  Hence
  $$\sum\limits_{j = 1}^v {{u_j}{\mu _0}( {D_j})}= \mathop {\lim }\limits_{k \to \infty }(\sum\limits_{j = 1}^v {{u_j}{\mu _{k}}( {D_j})}).$$
  we have
  \begin{equation}\label{E:4.3}
\int {\varphi d{\mu _{k}}}  \to \int {\varphi d{\mu _0}} .
\end{equation}
 Next, we take a sequence of step functions
   $\{ {f_n}\} $ so that $\operatorname{spt} {f_n} \subset \mathop \Pi \limits_{i = 1}^d[ - l,l)$, ${f_n} \to f$ uniformly and $\{ {f_n}\} $ is uniformly bounded.\\
   By Lebesgue's dominated convergence theorem and \eqref{E:4.3}, we have
$$\mathop {\lim }\limits_{n \to \infty } \mathop {\lim }\limits_{k \to \infty } \int {{f_n}d{\mu _{k}}}  = \mathop {\lim }\limits_{n \to \infty } \int {{f_n}d{\mu _0}}  = \int {\mathop {\lim }\limits_{n \to \infty } {f_n}d{\mu _0}}  = \int {fd{\mu _0}}.$$
   Hence we only need to prove
   $$\mathop {\lim }\limits_{k \to \infty } \int {fd{\mu _{k}}}  = \mathop {\lim }\limits_{n \to \infty } \mathop {\lim }\limits_{k \to \infty } \int {{f_n}d{\mu _{k}}} .$$
  We just need to prove
  \begin{equation}\label{E:4.4}
\mathop {\lim }\limits_{n \to \infty } (\mathop {\lim \sup}\limits_{k \to \infty }  \int_{\mathop \Pi \limits_{i = 1}^d[ - l,l)} {\left| {{f_n} - f} \right|d{\mu _{k}}) = } 0.
\end{equation}
   Since ${f_n} \to f$ uniformly, given any $\varepsilon>0 $, there exists a positive integer $N_1$ such that ${\left| {{f_n} - f} \right| < \varepsilon }$ for each $n > {N_1}$.
   Hence for each $n > {N_1}$,
   $$\int_{\mathop \Pi \limits_{i = 1}^d[ - l,l)} {\left| {{f_n} - f} \right|d{\mu _{k}} < } {\mu _{k}}(\mathop \Pi \limits_{i = 1}^d[ - l,l)) \cdot \varepsilon  = {{\cal H}^s}(({e^{{t_{{s_k}}}}}{C^x})  \cap \mathop \Pi \limits_{i = 1}^d[ - l,l)) \cdot \varepsilon .$$
By \eqref{E:4.1}, we have
${{\cal H}^s}(({e^{{t_{{s_k}}}}}{C^x} )\cap \mathop \Pi \limits_{i = 1}^d[ - l,l)) \to {{\cal H}^s}(F \cap \mathop \Pi \limits_{i = 1}^d[ - l,l))$.
By Lemma \ref{le:3.7} (1) and Claim 2 of Proposition \ref{pr:5.2}, we also have
${{\cal H}^s}(({e^{{t_{{s_k}}}}}{C^x}) \cap \mathop \Pi \limits_{i = 1}^d[ - l,l]) \to {{\cal H}^s}(F \cap \mathop \Pi \limits_{i = 1}^d[ - l,l])$.\\
Hence we have
$$\mathop {\lim \sup}\limits_{k \to \infty} \int_{\mathop \Pi \limits_{i = 1}^d[ - l,l)} {\left| {{f_n} - f} \right|d{\mu _{k}} \le } {{\cal H}^s}(F \cap \mathop \Pi \limits_{i = 1}^d[ - l,l]) \cdot \varepsilon $$
for each $n > {N_1}$.\\
So if ${{\cal H}^s}(F \cap \mathop \Pi \limits_{i = 1}^d[ - l,l])$ is finite, then \eqref{E:4.4} holds.\\
Since
$({e^{{t_{{s_k}}}}}{C^x}) \cap B(0,l)\subset ({e^{{t_{{s_k}}}}}{C^x}) \cap \mathop \Pi \limits_{i = 1}^d[ - l,l] \subset ({e^{{t_{{s_k}}}}}{C^x}) \cap B(0,\sqrt d  \cdot l)$,
${{\cal H}^s}(({e^{{t_{{s_k}}}}}{C^x}) \cap B(0,l)) \le {{\cal H}^s}(({e^{{t_{{s_k}}}}}{C^x}) \cap \mathop \Pi \limits_{i = 1}^d[ - l,l])\le {{\cal H}^s}(({e^{{t_{{s_k}}}}}{C^x} )\cap B(0,\sqrt d  \cdot l))$.
Hence we have 
$$\mathop {\lim  \sup}\limits_{k \to \infty } {{\cal H}^s}(({e^{{t_{{s_k}}}}}{C^x}) \cap B(0,l)) \le{{\cal H}^s}(F \cap \mathop \Pi \limits_{i = 1}^d[ - l,l])\le \mathop {\lim \inf}\limits_{k \to \infty }  {{\cal H}^s}(({e^{{t_{{s_k}}}}}{C^x}) \cap B(0,\sqrt d  \cdot l)).$$
Let ${r_k} = {e^{ - {t_{{s_k}}}}}$, by Definition \ref{de:2.7}, we have
    \begin{align*}
      \mathop {\lim \sup}\limits_{k \to \infty }  {{\cal H}^s}(({e^{{t_{{s_k}}}}}{C^x}) \cap B(0,l))
      =&\mathop {\lim \sup}\limits_{r_k \to 0}  {{\cal H}^s}(\frac{C^x}{{{r_k}}} \cap B(0,l)) \\
      =& \mathop {\lim \sup}\limits_{r_k \to 0}  {{{{(2l)}^s}{{\cal H}^s}(C \cap B(x,l{r_{k}}))} \over {{{(2l{r_{k}})}^s}}}\\
      \ge& {(2l)^s}\theta _*^s(C,x)
    \end{align*}
and
\begin{align*}
    \mathop {\lim \inf}\limits_{k \to \infty }  {{\cal H}^s}(({e^{{t_{{s_k}}}}}{C^x}) \cap B(0,\sqrt d  \cdot l))&=\mathop {\lim \inf}\limits_{{r_{k}} \to 0} {{\cal H}^s}(\frac{C^x}{{{r_k}}} \cap B(0,\sqrt d  \cdot l))\\
    &= \mathop {\lim \inf}\limits_{{r_{k}} \to 0}  {{{{(2\sqrt d  \cdot l)}^s}{{\cal H}^s}(C \cap B(x,\sqrt d  \cdot l{r_{k}}))} \over {{{(2\sqrt d  \cdot l{r_{k}})}^s}}}\\
    &\le {(2\sqrt d  \cdot l)^s}{\theta ^{*s}}(C,x).
\end{align*}
Hence we have
$${{\cal H}^s}(F \cap \mathop \Pi \limits_{i = 1}^d[ - l,l]) \in [{(2l)^s}\theta _*^s(C,x),{(2\sqrt d  \cdot l)^s}{\theta ^{*s}}(C,x)].$$
By Remark \ref{re:2.8}, ${{\cal H}^s}(F \cap \mathop \Pi \limits_{i = 1}^d[ - l,l])$ is finite.
This completes the proof of \eqref{E:4.11}. \\
Hence
  $${{\cal H}^s}\lfloor _{{e^{{t_{s_k}}}}{C^x}}\mathop  \to \limits^w {{\cal H}^s}\lfloor _F.$$
This completes the proof.\qed
\end{prof}

\begin{corollary}\label{co:4.3}
Let $C$ be as in Theorem \ref{th:4.1}. Let $\mu  = {{\cal H}^s}\lfloor_C$ with $s = {\dim _H}(C)$. Then for any $x \in C$, we have
$${\mathop{\rm Tan}\nolimits} (\mu ,x) \supset \{ {c{\cal H}^s}\lfloor_ F| F \in \omega({C^x}),\mbox{c}>0\}. $$
\end{corollary}
\begin{prof}
For any $F \in \omega({C^x})$, by Theorem \ref{th:4.1} and Lemma \ref{le:2.12}, we have ${{\cal H}^s}\lfloor_ F \in {\mathop{\rm Tan}\nolimits} (\mu ,x)$. \\
Hence by Definition \ref{de:2.10}, for any $c > 0$, ${c{\cal H}^s}\lfloor_ F \in {\mathop{\rm Tan}\nolimits} (\mu ,x)$.\\
Therefore, ${\mathop{\rm Tan}\nolimits} (\mu ,x) \supset \{ {c{\cal H}^s}\lfloor_ F| F \in \omega({C^x}),\mbox{c}>0\}.$\qed
\end{prof}\\
\textbf{Proof of Theorem \ref{th:1.1}} \quad 
 By Remark \ref{re:1.3}, we only need to consider the case where $C$ is not included in any hyperplane in ${{\mathbb{R}}^d}$.
 By Corollary \ref{co:4.3}, we only need to prove
  $${\mathop{\rm Tan}\nolimits} (\mu ,x) \subset \{ {c{\cal H}^s}\lfloor_ F|F \in \omega({C^x}),\mbox{c}> 0\}. $$
  For any $\nu \in {\mathop{\rm Tan}\nolimits} (\mu ,x)$, by Lemma \ref{le:2.12}, there exists $c>0$ and a sequence $ \{{r_k}\}$ such that ${r_k} \to 0$ as $k \to \infty$ and $c{{\cal H}^s}\lfloor_\frac{{{C^x}}}{{{r_k}}}\mathop  \to \limits^w \nu$.\\
  Since $\Gamma $ is sequentially compact (see Remark \ref{re:2.14}) and $\{\frac{{{C^x}}}{{{r_k}}}\}\subset \Gamma$, by Definition \ref{de:3.1} we know that
there exists $F  \in \omega({C^x})$ and a sub-sequence $\{ {r_{{j_k}}}\}$ of $\{ {r_k}\} $ such that
$$F = (\rho )\mathop {\lim }\limits_{k \to \infty }  \frac{{{C^x}}}{{{r_{j_k}}}} .$$
   By Theorem \ref{th:4.1}, there is a sub-sequence $\{ {r_{{s_k}}}\} $ of $\{ {r_{{j_k}}}\} $  such that
  $${{\cal H}^s}\lfloor_\frac{{{C^x}}}{{{r_{{s_k}}}}}\mathop  \to \limits^w {{\cal H}^s}\lfloor_ {F}.$$
  Since $c{{\cal H}^s}\lfloor_\frac{{{C^x}}}{{{r_{k}}}}\mathop  \to \limits^w \nu$, we have
  $$c{{\cal H}^s}\lfloor_\frac{{{C^x}}}{{{r_{{s_k}}}}}\mathop  \to \limits^w \nu.$$
  Because of the uniqueness of the weak convergence limit, $\nu=c{{\cal H}^s}\lfloor_ {F }$. Then we have
  $${\mathop{\rm Tan}\nolimits} (\mu ,x)\subset \{ {c{\cal H}^s}\lfloor_ F| F \in \omega({C^x}), c>0\}. $$
 This completes the proof.
\qed

\section{Proof of Theorem \ref{th:1.2}}
\begin{definition}\label{de:5.1}
	Let $\tau, s > 0$, a Borel measure $\mu$ on $\mathbb{R}^d$ is called \textbf{Ahlfors-David $s$-regular} (\textbf{$s$-regular} for short) if for all $x \in \operatorname{spt} \mu$ and $0 < r \leq \text{diam}(\operatorname{spt}\mu)$, the inequality
	$$
	\tau^{-1} \cdot r^s \leq \mu(B(x, r)) \leq \tau \cdot r^s
	$$
	holds, where $\text{diam}(\operatorname{spt} \mu)$ is the diameter of $\operatorname{spt}\mu$ and $B(x, r)$ is the closed ball with center $x$ and radius $r$.\\
	A closed set $G \subset \mathbb{R}^d$ is called \textbf{Ahlfors-David $s$-regular} (\textbf{$s$-regular} for short) if there exists a non-trivial $s$-regular measure $\mu$ such that $G = \operatorname{spt}\mu$.
\end{definition}

\begin{remark}\label{zu:5}
	(1). Assume the Borel measure $\mu$ is $s$-regular. Then there exists $\tau > 0$ such that for all $x \in \operatorname{spt} \mu$, we have
	$$
	0 < \tau^{-1} \cdot 2^{-s} \leq \theta_*^s(\mu, x) \leq \theta^{*s}(\mu, x) \leq \tau \cdot 2^{-s} < \infty.
	$$
	Furthermore, by (\cite{ref016}; Lemma 2.11), for any $s$-regular Radon measure $\mu$ on $\mathbb{R}^d$ and any $x \in \operatorname{spt} \mu$, we have
	$$
		\begin{aligned}
		\operatorname{Tan} (\mu ,x)= \{ \nu: \exists \; {r_i} \to {\rm{0}}, 0<c<\infty, c\mathop {{{\mu _{x,{r_i}}}} \over {{r_i}^s}}\mathop  \to \limits^w \nu\}.
		\end{aligned}
		$$
	(2). According to the mass distribution principle (\cite{ref17}; Proposition 2.2), a closed set $C$ is $s$-regular if and only if $\mathcal{H}^s\lfloor_C$ is $s$-regular. This can be regarded as an equivalent definition of $s$-regular sets. Additionally, it is evident that the Hausdorff dimension of any $s$-regular set is $s$ (see \cite{barany2023self}; Definition 2.1.4 and Proposition 2.1.5).\\
	(3). If a closed set $C$ is an $s$-regular set, then there exists $\tau > 0$ such that for all $x \in \operatorname{spt} \mu$,
	$$
	0 < \tau^{-1} \cdot 2^{-s} \leq \theta_*^s(C, x) \leq \theta^{*s}(C, x) \leq \tau \cdot 2^{-s} < \infty.
	$$
	(4). For any $s$-regular set $C$ in $\mathbb{R}^d$, $\mathcal{H}^s\lfloor_C$ is a Radon measure. This follows from the fact that $\mathcal{H}^s$ is Borel regular (\cite{ref10}; 4.5. Corollary), hence $\mathcal{H}^s\lfloor_C$ is also Borel regular. Since $\mathcal{H}^s\lfloor_C$ is $s$-regular, it is locally finite. Therefore, by (\cite{ref10}; 1.11. Corollary), $\mathcal{H}^s\lfloor_C$ is a Radon measure.\\
	(5). Any self-conformal set $F$ in $\mathbb{R}^d$ satisfying the weak separation condition is $s$-regular (\cite{angelevska2020self}; Theorem 3.1 and Proposition 3.3), where $s = \dim_H (F)$ and $\mu = \mathcal{H}^s\lfloor_F$. In particular, any self-similar set in $\mathbb{R}^d$ satisfying the open set condition and having Hausdorff dimension $s$ is also $s$-regular (\cite{barany2023self}; Theorem 4.1.2).
\end{remark}

\subsection{$\omega({C^x}) \supset \{\operatorname{spt}\nu|\nu  \in {\mathop{\rm Tan}\nolimits} (\mu ,x)\}$}

\begin{lemma}\label{le:5.1}
Given an $s$-regular Radon measure $\mu$ on $\mathbb{R}^d$ with $C = \operatorname{spt} \mu$. For any $x \in C$ and any $\nu \in \operatorname{Tan} (\mu, x)$, we have
	$
	\operatorname{spt}\nu \in \omega({C^x}).
	$
\end{lemma}

\begin{prof}
For each $x \in C$ and any $\nu \in {\mathop{\rm Tan}\nolimits} (\mu ,x)$, there exists a sequence $\{ {r_n}\}$ and a positive number $c$ such that ${r_n} \to 0$ and $ c\mathop {{{\mu _{x,{r_n}}}} \over {{r_n}^s}}\mathop  \to \limits^w \nu$ 
 by Remark \ref{zu:5} (1).
 That is, 
 $$\mathop {\lim }\limits_{n \to \infty } {{{\rm{ }}c} \over {{r_n}^s}}\int {fd{\mu _{x,{r_n}}}}=\mathop {\lim }\limits_{n \to \infty } {{{\rm{ }}c} \over {{r_n}^s}}\int_{{{{C^x}} \over {{r_n}}}} {fd{\mu _{x,{r_n}}}}  = \int {fd\nu }\; for \; all \; f \in {C_c}({{\mathbb{R}^d}}),$$
where ${C_c}({{\mathbb{R}^d}})$ denotes the space of compactly supported continuous real-valued functions on ${{\mathbb{R}^d}}$.\\
To prove $\operatorname{spt}{\nu}\in \omega({C^x}),$ we only need to show that  $${{{C^x}} \over {{r_n}}}\mathop  \to \limits^{\rho} \operatorname{spt}{\nu}.$$
Next, we will use Lemma \ref{le:3.4} to complete the proof of the following claim.\\
\textbf{Claim 1} If $ c\mathop {{{\mu _{x,{r_n}}}} \over {{r_n}^s}}\mathop  \to \limits^w \nu$ with $0<c<\infty$ and ${r_n} \to 0$ as $n \to \infty $, then  ${{{C^x}} \over {{r_n}}}\mathop  \to \limits^{\rho} \operatorname{spt}{\nu}.$\\
(i) For any $y\in \operatorname{spt}{\nu},$ then $\nu(B(y,r))>0$ for any $r>0$. We shall form a sequence $\{y_n\}$ satisfying $y_n\in {{{C^x}} \over {{r_n}}}$ and $\mathop {\lim }\limits_{n \to \infty }y_n=y$. For any $r>0$, we can choose a function $f\in {C_c}({{\mathbb{R}^d}})$ satisfying $\operatorname{spt} f\subset B(y,r)$ and $\int {fd\nu  = \int_{B(y,r)} {fd\nu  > 0} }$. Hence
$$\mathop {\lim }\limits_{n \to \infty } {{{\rm{ }}c} \over {{r_n}^s}}\int_{{{{C^x}} \over {{r_n}}}} {fd{\mu _{x,{r_n}}}}= \mathop {\lim }\limits_{n \to \infty } {{{\rm{ }}c} \over {{r_n}^s}}\int_{{{{C^x}} \over {{r_n}}}\cap B(y,r)} {fd{\mu _{x,{r_n}}}} >0.$$
So ${\mu _{x,{r_n}}}({{{C^x}} \over {{r_n}}}\cap B(y,r))>0$ when $n$ is large enough.
That is, ${{{C^x}} \over {{r_n}}} \cap B(y,r) \ne \emptyset $ when $n$ is large enough.\\
Let $r = {1 \over k}, k = 1,2, \cdots $, there exists a positive integer ${n_k}$ such that ${{{C^x}} \over {{r_n}}} \cap B(y,{1 \over k}) \ne \emptyset $ for $n \ge {n_k}$.
Thus it is possible to form a sequnence $\{y_n\}, n\ge {n_1}$, satisfying  ${y_n}\in {{{C^x}} \over {{r_n}}} \cap B(y,{1 \over k}) $ for $n=n_k, n_k+1, \ldots, n_{k+1}-1$, and this sequence clearly converges toward $y$.\\
(ii) For any subsequence ${\{ {n_k}\} _{k \in {\mathbb{N_+}}}}$ and any sequence ${\{y_{{n_k}}\} _{k \in {\mathbb{N_+}}}}$ with ${y_{{n_k}}} \in {{{C^x}} \over {{r_{n_k}}}}$, suppose  $\{y_{{n_k}}\}$ converges, let $y=\mathop {\lim }\limits_{k \to \infty } {y_{{n_k}}}$. 
Now we prove that
$y\in \operatorname{spt}{\nu}$. \\
If $y \notin \operatorname{spt}{\nu}$, there exists a positive number $r$ such that ${\nu}(B(y,r))=0$.  We choose a function $f\in {C_c}({{\mathbb{R}^d}})$ satisfying $\operatorname{spt} f\subset B(y,r)$ and
$f(x)=1$ for $x\in B(y,r/2)$. Then 
$\int {fd\nu  = \int_{B(y,r)} {fd\nu  =0} }$. 
Hence 
\begin{equation}
	\begin{aligned}\label{eq1} 
0&=\mathop {\lim }\limits_{k \to \infty }{{{\rm{ }}c} \over {{r_{n_k}}^s}} \int_{{{{C^x}} \over {{r_{n_k}}}}} {fd{\mu _{x,{r_{n_k}}}}} \\
&\ge \mathop {\lim \sup }\limits_{k \to 0} {{{\rm{ }}c} \over {{r_{n_k}}^s}}\int_{{{{C^x}} \over {{r_{n_k}}}}\cap B(y,r/2)} {fd{\mu _{x,{r_{n_k}}}}}\\
	& =\mathop {\lim \sup }\limits_{k \to 0}{{{\rm{ }}c} \over {{r_{n_k}}^s}} {\mu _{x,{r_{n_k}}}}({{{C^x}} \over {{r_{n_k}}}}\cap B(y,r/2)).
	\end{aligned}
\end{equation}
In fact, we have $B({y_{{n_k}}},r/4) \subset B(y,r/2)$ when $k$ is large enough since  $y=\mathop {\lim }\limits_{k \to \infty } {y_{{n_k}}}$.\\
Hence by \eqref{eq1}, we have
	\begin{equation}\label{eq2} 
		\mathop {\lim \sup }\limits_{k \to 0} {{{\rm{ }}c} \over {{r_{n_k}}^s}}{\mu _{x,{r_{n_k}}}}({{{C^x}} \over {{r_{n_k}}}}\cap B({y_{{n_k}}},r/4))=0.
	\end{equation}
Since ${y_{{n_k}}} \in {{{C^x}} \over {{r_{n_k}}}}$,
there exists ${x_{{n_k}}} \in C$ such that ${y_{{n_k}}} = {{{x_{{n_k}}} - x} \over {{r_{{n_k}}}}}$.
Moreover, since $\mu$ is $s$-regular, there exists $\tau > 0$ such that for all $x\in C$ and $0 < r\leq \text{diam}(C)$, the following inequalities hold:
$${\tau^{ - 1}}\cdot{r^s} \leq \mu (B(x,r)) \leq \tau \cdot{r^s}.$$
Hence 
\begin{equation}
		\begin{aligned}
			{{{\rm{ }}c} \over {{r_{n_k}}^s}}{\mu _{x,{r_{n_k}}}}({{{C^x}} \over {{r_{n_k}}}}\cap B({y_{{n_k}}},r/4))&={{c{\mu}(C \cap B({x_{{n_k}}},r{r_{{n_k}}}/4))} \over {^{{r_{{n_k}}}^s}}}\\
			& =
			{c{\mu}(B({x_{{n_k}}},r{r_{{n_k}}}/4))\over {^{{r_{{n_k}}}^s}}}\\
			&\ge c\tau^{-1}{({{r} \over {4}})^s}.
		\end{aligned}
	\end{equation}
Since $c\tau^{-1}{({{r} \over {4}})^s}>0$, this contradicts equation \eqref{eq2}.
Hence $y\in \operatorname{spt}{\nu}$. \\
To sum up, by Lemma \ref{le:3.4}, we have
${{{C^x}} \over {{r_n}}}\mathop  \to \limits^{\rho} \operatorname{spt}{\nu}.$\qed
\end{prof}

\subsection{$\omega({C^x}) = \{\operatorname{spt} \nu|\nu  \in {\mathop{\rm Tan}\nolimits} (\mu ,x)\}$}

\begin{lemma}[\cite{ref10}; 1.23.Theorem]\label{le:5.3}
If ${\mu _1},{\mu _2}, \cdots $ are Radon measures on ${{\mathbb{R}^d}}$ with $$\mathop {\sup }\limits_{i \ge 1} \{ {\mu _i}(K)\}  < \infty $$ for all compact sets $K\subset{{\mathbb{R}^d}}$, then there is a weakly convergent subsequence of $\{ {\mu _i}\} $. 
\end{lemma}

\begin{lemma}\label{le:5.4}
Given an $s$-regular Radon measure $\mu$ on $\mathbb{R}^d$ with $C = \operatorname{spt}\mu$.
Then for any $x\in C$ and any $F\in\omega(C^x)$, there exists $\nu\in\mathop{\rm Tan}\nolimits(\mu,x)$ such that
$F = \operatorname{spt}\nu.$
\end{lemma}
\begin{prof}
For any 
$F \in \omega({C^x})$, by Definition \ref{de:3.1}, there exists a sequence 
$\{r_n\}$ such that 
$$ F = (\rho )\mathop {\lim }\limits_{n \to \infty }{{{C^x}} \over {{r_n}}}.$$
Suppose $\{\mathop {{{\mu _{x,{r_n}}}} \over {{r_n}^s}}\}$ has a weakly convergent subsequence, by Remark \ref{zu:5} (1),
then there exists a subsequence $\{r_{{n_k}} \}\subset\{r_n\}$ and $\nu \in {\mathop{\rm Tan}\nolimits} (\mu ,x)$ such that 
	$\mathop {{{\mu _{x,{r_{n_k}}}}} \over {{r_{n_k}}^s}}\mathop  \to \limits^w \nu.$
By Claim 1 of Lemma \ref{le:5.1}, we have ${{{C^x}} \over {{r_{n_k}}}}\mathop  \to \limits^{\rho} \operatorname{spt}{\nu}$.\\
Since $ F = (\rho )\mathop {\lim }\limits_{n \to \infty }{{{C^x}} \over {{r_n}}}$, we have ${{{C^x}} \over {{r_{n_k}}}}\mathop  \to \limits^{\rho} F$.
Because of the uniqueness of the convergence limit by topology $\rho$, we obtain that $\operatorname{spt}\nu=F$.\\
Hence we only need to prove that
 $\{\mathop {{{\mu _{x,{r_n}}}} \over {{r_n}^s}}\}$ has a weakly convergent subsequence to obtain Lemma \ref{le:5.4}, see the following claim.\\
\textbf{Claim 1}
For any $F= (\rho )\mathop {\lim }\limits_{n \to \infty }{{{C^x}} \over {{r_n}}} \in \omega({C^x})$, then $\{\mathop {{{\mu _{x,{r_n}}}} \over {{r_n}^s}}\}$ has a weakly convergent subsequence.\\
Indeed, according to Lemma \ref{le:5.3}, we just need to prove that 
\begin{equation}\label{eq:5.5} 
\sup_{n\geq1}\left\{\frac{\mu_{x,r_n}}{r_n^s}(K)\right\}<\infty
\end{equation}
for all compact sets $K\subset{{\mathbb{R}^d}}$.\\
Below we provide the proof of \eqref{eq:5.5}.\\
For any compact set $K\subset{{\mathbb{R}^d}}$, there exists a positive number $l$ such that 
$K\subset B(0,l)$.  
Since $\mu$ is $s$-regular, there exists $\tau>0$ such that for all $x\in C$ and $0 < r\le diam(C)$, it holds that
$${\tau^{ - 1}}\cdot{r^s} \le \mu (B(x,r)) \le \tau \cdot{r^s}.$$ 
	Hence
	\begin{align*}
		\mathop {{{\mu _{x,{r_n}}}} \over {{r_n}^s}}(K) &\le \mathop {{{\mu _{x,{r_n}}}} \over {{r_n}^s}}(B(0,l))  \\
		&  = \mathop {{\mu(B(x,{r_n}l))} \over {{r_n}^s}} \\
		&  \le \tau \cdot{l^s}.
	\end{align*}
  So we have
  $\mathop {\sup }\limits_{n \ge 1} \{ \mathop {{{\mu _{x,{r_n}}}} \over {{r_n}^s}}(K)\}  \le \tau \cdot{l^s}< \infty$. The inequality
 \eqref{eq:5.5} is proved.\\
  This completes the proof of Lemma \ref{le:5.4}.\qed
\end{prof}

\textbf{Proof of Theorem \ref{th:1.2}}
\quad On the one hand, by Lemma \ref{le:5.1}, we know that
$$\{ \operatorname{spt}\nu|\nu  \in {\mathop{\rm Tan}\nolimits} (\mu ,x)\}  \subset \omega({C^x}).$$
On the other hand, by Lemma \ref{le:5.4}, we know that
$$\omega({C^x}) \subset \{ \operatorname{spt} \nu|\nu  \in {\mathop{\rm Tan}\nolimits} (\mu ,x)\}.$$
Hence $\omega({C^x})=\{ \operatorname{spt} \nu|\nu  \in {\mathop{\rm Tan}\nolimits} (\mu ,x)\}$.
\qed

\begin{corollary}\label{co:1.3}
Given any $s$-regular set $C$, let $\mu = {{\cal H}^s}\lfloor_C$. Then for any $x\in C$, we have
\begin{align}\label{E：6.2.3}
\omega({C^x})=\{\operatorname{spt} \nu|\nu\in {\mathop{\rm Tan}\nolimits} (\mu,x)\}.
\end{align}
\end{corollary}
\begin{prof}
It can be directly deduced from Remark \ref{zu:5}(3)-(4) and Theorem \ref{th:1.2}.
\qed
\end{prof}

\begin{remark}
Corollary \ref{co:1.3} applies to some fractals that satisfy $s$-regular, such as self-similar sets satisfying the open set condition and self-conformal sets satisfying the weak separation condition.
\end{remark} 

  
\section*{Declaration of competing interest} 
The author declares that the publication of this paper has no conflict of interest.

\section*{Data availability} 
No data was used for the research described in the article.

\end{document}